\documentclass[reqno,12pt]{amsart}
\usepackage{amsmath,epsfig}
\usepackage{amssymb}
\usepackage[latin1]{inputenc}
\usepackage{lmodern}
\usepackage{supertabular}
\usepackage{caption}
\usepackage{geometry}
\usepackage{ulem}
\newtheorem{theorem}{Theorem}[section]
\newtheorem{lemma}[theorem]{Lemma}

\newtheorem{cor}[theorem] {Corollary}

\theoremstyle{definition}

\newtheorem{remark}[theorem] {Remark}

\renewcommand{\section}{\secdef\sct\sect}
\newcommand{\sct}[2][default]{\refstepcounter{section}
\setcounter{equation}{0}
\vspace{0.5cm}
\centerline{ \large
\scshape \arabic{section}.\ #1}
\vspace{0.3cm}}
\newcommand{\sect}[1]{
\vspace{0.5cm}
\centerline{\large\scshape #1}
\vspace{0.3cm}}

\renewcommand{\subsection}{\secdef \subsct\sbsect}
\newcommand{\subsct}[2][default]{\refstepcounter{subsection}
\nopagebreak
\vspace{0.5\baselineskip}
{\flushleft\bf \arabic{section}.\arabic{subsection}~\bf #1  }
\nopagebreak}
\newcommand{\sbsect}[1]{\vspace{0.1cm}\noindent
{\bf #1}\vspace{0.1cm}}

\newcommand{\Ll}{\mathcal{L}}

\newcommand{\F}{{\mathcal {F}}}

\newcommand{\Oo}{{\mathcal {O}}}

\newcommand{\R}     {\mathbb{R}}

\newcommand{\N}     {\mathbb{N}}

\newcommand{\E}     {\mathbb{E}}

\newcommand{\V}     {\mathbb{V}}

\newcommand{\cov}{\text{Cov}}

\newcommand{\Mm}{\mathcal{M}}

\def\1{{\mathchoice {1\mskip-4mu\mathrm l}
                    {1\mskip-4mu\mathrm l}
                    {1\mskip-4.5mu\mathrm l} {1\mskip-5mu\mathrm l}}}

\newcommand{\abs}[1]{\ensuremath{\left\vert#1\right\vert}}
\begin{document}
\renewcommand{\proofname}{Proof}
%\today
\title[Rates of convergence in the Blume-Emery-Griffiths model]{\large
Phase transitions \\\vspace{2mm} for rates of convergence \\\vspace{3mm} in the Blume-Emery-Griffiths model}

\author[Peter Eichelsbacher and Bastian Martschink]{}
\maketitle
\thispagestyle{empty}
\vspace{0.2cm}

\centerline{\sc Peter Eichelsbacher\footnote{Ruhr-Universit\"at Bochum, Fakult\"at f\"ur Mathematik,
NA 3/66, D-44780 Bochum, Germany, {\tt peter.eichelsbacher@rub.de}} and Bastian Martschink\footnote{Hochschule Bonn-Rhein Sieg, Fachbereich 03,
B 295, D-53757 Sankt Augustin, Germany, {\tt bastian.martschink@h-brs.de} \\The authors have been supported by Deutsche Forschungsgemeinschaft via SFB/TR 12.}
}

\vspace{2 cm}

%\centerline{\small{\version}}
%\vspace{1cm}

\begin{quote}
{\small {\bf Abstract:} 
We derive rates of convergence for limit theorems that reveal the intricate structure of the phase transitions in a mean-field version of the Blume-Emery-Griffith model. The theorems consist of scaling limits for the total spin. The model depends on the inverse temperature $\beta$ and the interaction strength $K$.
The rates of convergence results are obtained as $(\beta,K)$ converges along appropriate sequences $(\beta_n,K_n)$ to points belonging to various subsets
of the phase diagram which include a curve of second-order points and a tricritical point. We apply Stein's method for normal and non-normal
approximation avoiding the use of transforms and supplying bounds, such as those of Berry-Esseen quality, on approximation error. We observe an additional
phase transition phenomenon in the sense that depending on how fast $K_n$ and $\beta_n$ are converging to points in various subsets of the phase diagram,
different rates of convergences to one and the same limiting distribution occur.}
\end{quote}

%\vfill



\setcounter{section}{0}

\bigskip
\bigskip
\bigskip

\section{Introduction}
\subsection{The Blume-Emery-Griffiths Model}\label{subsubsec:IntroductionBME}

In 1971 Blume, Emery and Griffiths \cite{Blume/Emery/Griffiths:1971} introduced a mean field version of an important lattice spin model due to Blume and Capel. We refer to the mean field version as the BEG model. The BEG model is equivalent to the Blume-Capel model (see \cite{Blume:1966} and \cite{Capel:1966}, \cite{Capel:1967a} and \cite{Capel:1967b}) on the complete graph on $n$ vertices. One of the most outstanding features of the model is that it is one of the few mean-field models that exhibits a continuous second-order phase transition, a discontinuous first-order phase transition and thus has a tricritical point, which separates the curves of the points belonging to the phase transitions. As a consequence this model is used to study many diverse systems, obviously including the one Blume, Emery and Griffiths devised it for.\\
They showed that the model can be used to determine the phase diagram of He$^3$-He$^4$ mixtures using a simplification. 
In order to analyse this physical system the BEG model was introduced, which can also be used to explain the behavior of other physical systems such as microemulsions, semiconductor alloys or solid-liquid-gas systems, to name only a few. A variety of these applications of the BEG model are discussed in \cite[Section 1]{Ellis/Otto/Touchette:2005}. Especially because the model keeps the intricate phase transition structure 
it continues to be of interest in statistical mechanics.

Next we will give a mathematical definition of the BEG model and state some of the results known. 

Let $\beta >0$ and $K>0$. As a configuration space for the model we will take all the sequences $(\omega_1,\ldots,\omega_n)$ in $\{-1,0,1\}^n$. $\omega_i$ denotes the spin on site $i$ of the complete graph on $n$ vertices. The Hamiltonian for the BEG model is defined by
\begin{align}\label{HBEG}
H_{n,K}(w)=\sum\limits_{j=1}^n\omega_j^2-\frac{K}{n}\left(\sum\limits_{j=1}^n\omega_j\right)^2
\end{align}
for each $\omega\in\{-1,0,1\}^n$. $K>0$ represents the interaction strength of the model.  Given this Hamiltonian the probability of observing a subset $A$ of $\{-1,0,1\}^n$ equals
\begin{eqnarray}
P_{\beta,K,n}(A)=\frac{1}{Z_{ \beta,K,n}}\int_A\exp\bigl(-\beta H_{n,K}\bigr)\text{d}P_n\label{PBEG}.
\end{eqnarray}
$Z_{\beta,K,n}$ denotes the normalisation constant and $P_n$ is the product measure on $\{-1,0,1\}^n$ having identical one-dimensional marginals
$%\begin{align*}
\rho=\frac{1}{3}(\delta_{-1}+\delta_0+\delta_{+1})$.
%\end{align*}
We will be interested in the behavior of the spin per site
\begin{align}\label{totalspin}
\frac{1}{n}S_n(\omega):=\frac{1}{n}\sum\limits_{j=1}^n\omega_j
\end{align}
under the distribution $P_{ \beta,K,n}$ as $n\rightarrow \infty$. $S_n$ is called the {\it total spin}. The BEG model shares the feature - with for example the Curie-Weiss model - that the interaction terms in the Hamiltonian can be written as a quadratic function of the total spin. For this purpose we absorb the first non-interacting part of the Hamiltonian \eqref{HBEG} into the product measure. It is important to notice that in that case, in contrast to the Curie-Weiss model, the BEG model has a much more complicated product measure $P_{n,\beta}$ on $\{-1,0,1\}^n$ because of its $\beta$-dependence. The one-dimensional marginals of $P_{n,\beta}$ are given by
\begin{align*}
\rho_{\beta}(\text{d}\omega_j)=\frac{1}{Z(\beta)}\cdot\exp(-\beta\omega_j^2)\rho(\text{d}w_j),
\end{align*}
where $Z(\beta)$ is equal to 
$%\begin{align*}
\int\exp(-\beta\omega_j^2)\rho(\text{d}\omega_j)=\frac{1+2e^{-\beta}}{3}$.
%\end{align*}
Thus one has (see \cite[Section 3.1]{Ellis/Otto/Touchette:2005}) that the probability of observing a configuration $\omega$ equals
\begin{eqnarray}
P_{\beta,K,n}(\text{d}\omega)=\frac{1}{\tilde Z_{ \beta,K,n}}\cdot\exp\left[n\beta K\left(\frac{S_n(\omega)}{n}\right)^2\right]P_{n,\beta}(\text{d}\omega)\label{PBEG2}
\end{eqnarray}
with normalization constant $\tilde Z_{ \beta,K,n} = \frac{Z_{\beta, K, n}}{Z(\beta)^n}$. Hence one has reduced the BEG model to a Curie-Weiss-type model.

We appeal to the theory of large deviations to define the set of (canonical) {\it equilibrium macrostates}. In order to state a large deviations principle (LDP) (for a definition see \cite[Section 1.2]{Dembo/Zeitouni:LargeDeviations}) for the spin per site for the BEG model we need to define the cumulant generating function of $\rho_{\beta}$, which is given by
\begin{equation}
c_{\beta}(t)=\log\int\exp(t\omega_1)\rho_{\beta}(\text{d}\omega_1)
=\log\left(\frac{1+e^{-\beta}(e^t+e^{-t})}{1+2e^{-\beta}}\right).\label{CBEG}
\end{equation}
Cram\'er's theorem (\cite[Theorem 2.2.3]{Dembo/Zeitouni:LargeDeviations}) states that, with respect to the product measure $P_{n, \beta}$, the sequence $(S_n/n)_n$
satisfies the LDP on $[-1,1]$ with speed $n$ and rate function 
\begin{align}\label{LegFen}
J_{\beta}(z):=\sup\limits_{t\in\R}\{tz-c_{\beta}(t)\},
\end{align} 
which is the Legrende-Fenchel transform of $c_{\beta}$. Having the LDP for $(S_n/n)_n$ with respect to $P_{n,\beta}$, the following theorem, taken from \cite[Theorem 2.4]{EllisHaven:2000}, states the LDP for $(S_n/n)_n$ for $P_{\beta,K,n}$.

\begin{theorem}
For all $\beta>0$ and $K>0$ the following conclusion holds: With respect to $P_{\beta,K,n}$, $(S_n/n)_n$ satisfies the LDP on $[-1,1]$ with speed $n$ and rate function
\begin{align*}
I_{\beta,K}(z)=J_{\beta}(z)-\beta Kz^2-\inf\limits_{y\in\R}\{J_{\beta}(y)-\beta Ky^2\},
\end{align*}
with $J_{\beta}(z)$ taken from \eqref{LegFen}.
\end{theorem}

As a consequence only the points $z\in[-1,1]$ satisfying $I_{\beta,K}(z)=0$ do not have an exponentially small probability of being observed. These points form the set of the so-called {\it equilibrium macrostates}, which is accordingly defined as
\begin{align}\label{MacroBEG}
\Mm_{\beta,K}=\bigl\{z\in[-1,1]: I_{\beta,K}(z)=0\bigr\}.
\end{align}
In \cite[Theorem 3.6, Theorem 3.8]{Ellis/Otto/Touchette:2005} it is proven that there exists a critical inverse temperature $\beta_c=\log 4$
and, for $\beta >0$, there exists a critical value $K_c(\beta)>0$ characterising the phase-transition structure of the model: % (see Theorem \ref{phase}):
for $\beta>0$ and $0<K<K_c(\beta)$, $\Mm_{\beta,K}$ consists of the unique pure phase 0, \cite[Theorem 3.6(a) and 3.8(a)]{Ellis/Otto/Touchette:2005}. For $\beta>0$ and $K > K_c(\beta)$, $\Mm_{\beta,K}$ consists of two distinct, nonzero phases. For $0 < \beta \leq \beta_c$, as $K$ increases through $K_c(\beta)$, $\Mm_{\beta,K}$ undergoes a continuous bifurcation, which
corresponds to a second-order phase transition, \cite[Theorem 3.6(b)(c)]{Ellis/Otto/Touchette:2005}. Here we have
\begin{equation} \label{KcBEG}
K_c(\beta) = \frac{1}{2 \beta c_{\beta}''(0)} = \frac{e^{\beta}+2}{4\beta}.
\end{equation}
For $\beta > \beta_c$, as $K$ increases through $K_c(\beta)$, $\Mm_{\beta,K}$ undergoes a discontinuous bifurcation, which corresponds to a first-order phase transition, \cite[Theorem 3.8(c)(d)]{Ellis/Otto/Touchette:2005}. The point $(\beta_c, K_c(\beta_c)) = (\log 4, 3/ [2 \log 4])$ in the positive quadrant of the $\beta$-$K$ plane separates the second-oder phase transition from the first-order transition and is called the {\it tricritical point}.

Based on the points that correspond to the transitions we define two different sets that will influence the form of our limiting density. The first set contains the {\it single-phase region} and it is defined by
\begin{align*}
A=\bigl\{(\beta,K)\in\R^2:\,0<\beta\leq\beta_c,\,0<K<K_c(\beta)\bigr\}.
\end{align*}
The curve containing the {\it second-order points} is given by
\begin{align*}
B=\bigl\{(\beta,K)\in\R^2:\,0<\beta<\beta_c,\,K=K_c(\beta)\bigr\}.
\end{align*}
\begin{center}
\includegraphics[scale=0.37]{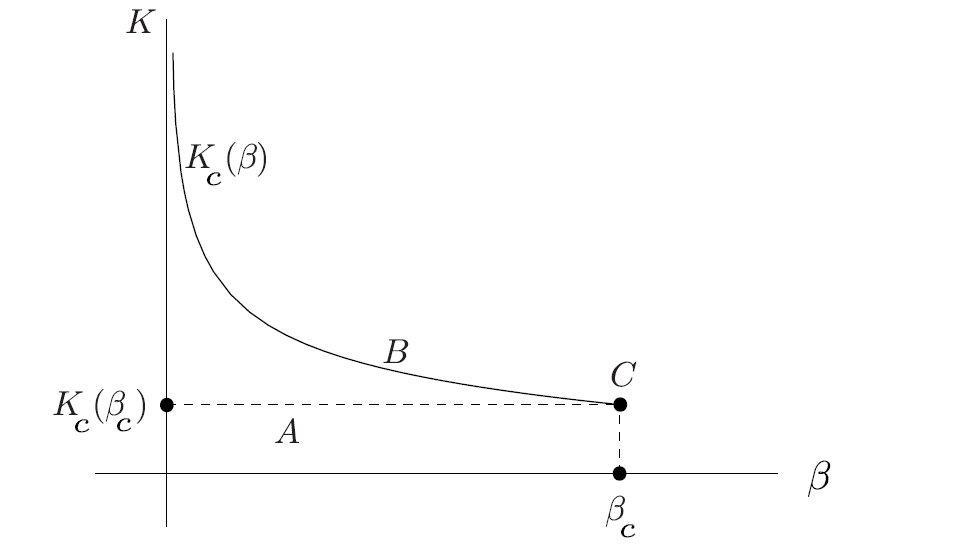}\\Figure 1.1: The sets A, B and C
\end{center}
Furthermore we consider the singleton set 
\begin{align*}
C:=(\beta_c,K_c(\beta_c))=(\log 4, 3/[2\log 4]),
\end{align*}
which separates the first- and second-order phase transition. See also \cite[Section 1]{CosteniucEllis:2007} for a nice summary. 
Figure 1.1 illustrates the sets.

When studying the {\it law of large numbers} for the BEG model this complex phase-transition structure proves to be a determining factor. In \cite[(2.1)]{CosteniucEllis:2007} it was shown that for $\beta>0$ and $0<K<K_c(\beta)$ the law of large numbers holds with
$
P_{\beta,K,n}\bigl(S_n/n\in \text{d}x\bigr)\Rightarrow\delta_0$,
as $n\rightarrow\infty$. Hence, for sufficiently small interaction strength $K>0$ an analogue of the classical law of large numbers can be proven. If the interaction strength exceeds the critical value for $K$ the law of large numbers breaks down. For $\beta>0$ and $K>K_c(\beta)$ in 
\cite[Theorem 3.6, 3.8]{Ellis/Otto/Touchette:2005} the existence of $z(\beta,K)>0$ was proven such that the following limit holds true:
\begin{align*}
P_{\beta,K,n}\bigl(S_n/n\in \text{d}x\bigr)\Rightarrow\frac{1}{2}\bigl(\delta_{z(\beta,K)}+\delta_{-z(\beta,K)}\bigr),
\end{align*}
see \cite[(2.2)]{CosteniucEllis:2007}.
Because of the intricate phase transition structure there are also two limits for $K=K_c(\beta)$. Whereas for $0<\beta\leq\beta_c$ the law of large numbers holds, $P_{\beta,K_c(\beta),n}\bigl(S_n/n\in \text{d}x\bigr)\Rightarrow\delta_0$,
for $\beta>\beta_c$ the limit is expressed by a measure supported at three points corresponding to the macrostates in $\Mm_{\beta,K}$:
\begin{align*}
P_{\beta,K_c(\beta),n}\bigl(S_n/n\in \text{d}x\bigr)\Rightarrow\lambda_0\delta_0+\lambda_1\bigl(\delta_{z(\beta,K_c(\beta))}+\delta_{-z(\beta,K_c(\beta))}\bigr),
\end{align*}
where $\lambda_0$ and $\lambda_1$ are positive numbers satisfying $\lambda_0+2\lambda_1=1$ (for an explicit display see \cite[(4.4)]{CosteniucEllis:2007}). These first hints of the intricacy of the phase-transition structure can also be seen for the limit theorems stated in Section \ref{subsec:Statement of Results BME}. 

In Section \ref{subsec:Statement of Results BME} we will obtain limit theorems and {\it rates of convergence} for the rescaled total spin $S_n/n^{1-\gamma}$ for appropriate choices of $\gamma\in(0,1/2]$. In \cite{CosteniucEllis:2007} 18 scaling limits and 18 moderate deviation principles for the total spin
$S_n$ were obtained as $(\beta,K)$ converges along appropriate sequences $(\beta_n,K_n)$ to points belonging to the three separate classes: (1) the tricritical point $C$, (2) the curve $B$ of second-order points, and (3) the single-phase region $A$ lying under the curve.  We obtain the 18 different scaling limits by
an alternative proof (Stein's method) and present rates of convergence in all 18 cases at the same time. Furthermore we observe that the complex structure of the phase transitions in the BEG model provides an additional insight, presenting that in 15 of 18 cases the rate of convergence differs within the same case.
A fixed case - out of the 18 - is characterised by a fixed limiting distribution. We will observe that for any of the 15 different limiting densities the rate of convergence
depends on the choice of the value of $\gamma$ and/or the choices of two more parameters $\Delta_1$ and $\Delta_2$ given as follows:
We consider sequences $(\beta_n,K_n)$ converging to $(\beta,K)$ taken from $A$, $B$ or $C$ along the sequences
\begin{align}\label{BetanBEG}
\beta_n&=\log\left(e^{\beta_c}-\frac{b}{n^{\Delta_1}}\right),\\\label{KnBEG}
K_n&=K_c(\beta_n)-\frac{k}{n^{\Delta_2}},
\end{align}
where $\Delta_1>0$, $\Delta_2>0$, $b\not= 0$, $k\neq 0$ and $K_c(\beta)$ defined in \eqref{KcBEG} for $\beta >0$. If $(\beta,K)$ is taken from a set $A$, $B$ or $C$ the sequence $(\beta_n,K_n)$ will determine the value of $\gamma$, since the sequence establishes which set influences the convergence towards $(\beta,K)$. Depending on the sign of $b$ and $k$ the sequences converge from regions having a different physical behavior. 
The mathematical explanation for the choices \eqref{BetanBEG} and \eqref{KnBEG} will be clear later (the sequences are chosen so that certain terms in a Taylor expansion have appropriate behaviour), whereas the physical significance is not obvious.

The three seeds from which the present paper grew are references \cite{Ellis/Otto/Touchette:2005}, \cite{CosteniucEllis:2007} and \cite{Eichelsbacher/Loewe:2010}. In the first paper the phase transition structure of the BEG model is analysed. In the second paper limit theorems in the BEG model are proven and in the third paper, rates of convergence are obtained for limit theorems for the Curie-Weiss model when the inverse temperature converges to the critical inverse temperature in the model along an appropriate sequence $\beta_n$. These results generalise the limit theorems obtained in \cite{Ellis/Newman:1978}
and \cite{Ellis/Newman/Rosen:1980}. The third paper developed Stein's method for exchangeable pairs for distributional approximations including the Gaussian distributions as well as non-Gaussian limit distributions and obtained convergence rates for the Curie-Weiss model
(see also \cite{Chatterjee/Shao:2011}). Note that the fact that limit results are obtained as $\beta$ converges along appropriate sequences $\beta_n$ is shared
with a number of mean field models, including the Curie-Weiss models (\cite{Eichelsbacher/Loewe:2004b}, \cite{Eichelsbacher/Loewe:2010}) and the
Hopfield model of spin glasses and neutral networks (\cite{Gentz/Loewe:1999}, \cite{Eichelsbacher/Loewe:2004}).

\subsection{The Function $G_{\beta,K}$ and its Properties}\label{subsec: GFunkBEG}

From now on we denote for a function $f:\R\rightarrow\R$ the $i$-th derivative by $f^{(i)}$. A crucial element for Stein's method is the function
\begin{align}\label{GBEG}
G_{\beta,K}(x)=\beta Kx^2-c_{\beta}(2\beta Kx)
\end{align}
for $x\in\R$ and its minima. $c_{\beta}$ denotes the cumulant generating function of $\rho_{\beta}$ given in \eqref{CBEG}. The function $G_{\beta,K}$ plays a central role in nearly every aspect of the analysis of the BEG model, since it gives an alternative characterisation of the set of equilibrium macro states ${\mathcal M}_{\beta,K}$.  Apart from being helpful while developing Stein's method in the sequel, its usefulness is also certain for the study of, for example, the phase transitions, the LDP or moderate deviations for the total spin per site, see \cite{CosteniucEllis:2007}. The fact that a wide variety of phenomena can be obtained via properties of a single function is an appealing feature which is shared with a number of other mean-field models including the Curie-Weiss model,
the Curie-Weiss-Potts model \cite{Ellis/Wang:1990} and the Hopfield model \cite{Figotin/Pastur:1977}.
 The next Lemma, proven in \cite[Proposition 3.4]{Ellis/Otto/Touchette:2005}, draws a connection between the equilibrium macrostates defined in \eqref{MacroBEG} and the minima of the function $G_{\beta,K}$.
 \begin{lemma} \label{aquMG}
 For each $x\in\R$ we define $G_{\beta,K}(x)$ as in \eqref{GBEG}. Then for each $\beta>0$ and $K>0$,
 \begin{align*}
 \min\limits_{\abs{x}\leq 1}\bigl\{J_{\beta}(x)-\beta Kx^2\bigr\}=\min\limits_{x\in\R}\bigl\{G_{\beta,K}(x)\bigr\},
 \end{align*}
 with $J_{\beta}(x)$ defined in \eqref{LegFen}. Additionally the global minimum points of $J_{\beta}(x)-\beta Kx^2$ coincide with the global minimum points of $G_{\beta,K}$ and thus
 \begin{align*}
 \Mm_{\beta,K}=\bigl\{x\in\R:\,x\text{ minimizes }G_{\beta,K}(x)\bigr\}.
 \end{align*}
 \end{lemma} 

With the help of Lemma \ref{aquMG} the structure of these minima was discussed and proven in \cite[Theorem 3.6, 3.8]{Ellis/Otto/Touchette:2005}.
For $\beta>0$ and $0<K<K_c(\beta)$, $\Mm_{\beta,K}=\{0\}$, for $0<\beta\leq\beta_c$ and $K=K_c(\beta)$, $\Mm_{\beta,K}=\{0\}$.
For $0<\beta<\beta_c$ and $K>K_c(\beta)$, there exists $z(\beta,K)$ such that $\Mm_{\beta,K}=\{\pm z(\beta,K) \}$.
For $\beta>\beta_c$ and $K=K_c(\beta)$, there exists $z(\beta,K_c(\beta))$ such that $\Mm_{\beta,K}=\{0,\pm z(\beta,K_c(\beta)) \}$.
Finally for $\beta>\beta_c$ and $K>K_c(\beta)$, there exists $z(\beta,K_c(\beta))$ such that $\Mm_{\beta,K}=\{\pm z(\beta,K_c(\beta)) \}$.

A crucial element for the analysis of the model is the Taylor expansion of \eqref{GBEG}. For general $(\beta,K)$ we have, since $G_{\beta,K}$ is real analytic, that for the global minimum point $0$
\begin{align*}
G_{\beta,K}(x)=G_{\beta,K}(0)+\frac{G_{\beta,K}^{(2r)}(0)}{(2r)!}x^{2r}+\Oo\left(x^{2r+1}\right)\quad \text{ as } x\rightarrow 0,
\end{align*}
since $G^{(1)}_{\beta,K}(0)=0$. Here, $r$ denotes the {\it type} of the global minimum point. In \cite[Theorem 4.2]{CosteniucEllis:2007} the types were determined in the following theorem.
\begin{theorem}\label{type}
For all $(\beta,K)\in A\cup B\cup C$, $\Mm_{\beta,K}=\{0\}$.
\begin{enumerate}
\item For all $(\beta,K)\in A$, the global minimum point $0$ has type $r=1$.
\item For all $(\beta,K_c(\beta))\in B$, the global minimum point $0$ has type $r=2$.
\item For all $ C=(\beta_c,K_c(\beta_c))$, the global minimum point $0$ has type $r=3$.
\end{enumerate}
\end{theorem}
\begin{remark}
For the values of the parameters that were not considered in Theorem \ref{type} the type of the global minimum points is 1, which is proven in \cite[Theorem 6.3]{Ellis/Otto/Touchette:2005}.
\end{remark}
This theorem will yield the Taylor expansion of $G_{\beta,K}$ if $(\beta,K)$ is fixed and taken from one of the sets $A$, $B$ or $C$. Next we deal with the sequences \eqref{BetanBEG} and \eqref{KnBEG} and the associated function $G_{\beta_n,K_n}$. The following theorem yields three different forms of the Taylor expansion of $G_{\beta_n,K_n}$, see \cite[Theorem 4.3]{CosteniucEllis:2007}.

\begin{theorem}
For $\gamma\in\R_+$ and for a positive bounded sequence $(\beta_n,K_n)$ the following conclusions hold. Let for $R>0$, $\abs{x}<R$.
\begin{enumerate}
\item For $(\beta_n,K_n)\rightarrow (\beta,K)\in A$ the type of the minimum point $0 \in {\mathcal M}_{\beta,K}$ is $r=1$ and there exists $\xi=\xi(x)\in [-x,x]$ such that
\begin{eqnarray}\label{taylorGBEGA}
G_{\beta_n,K_n}(x)=\frac{G_{\beta_n,K_n}^{(2)}(0)}{2}x^2+A_{\beta_n,K_n}(\xi(x))x^3.
\end{eqnarray}
The error terms $A_{\beta_n,K_n}(\xi(x))$ are uniformly bounded over $n\in\N$ and $x\in(-R,R)$. 
We have 
$G_{\beta_n,K_n}^{(2)}\left(0\right)=\frac{(2\beta_n K_n)(e^{\beta_n} +2 - 4 \beta_n K_n)}{e^{\beta_n}+2}$.
\item For $(\beta_n,K_n)\rightarrow (\beta,K_c(\beta))\in B$ the type of the minimum point  $0 \in {\mathcal M}_{\beta, K_c(\beta)}$
is $r=2$ and there exists $\xi=\xi(x)\in [-x,x]$ such that
\begin{eqnarray}\label{taylorGBEGB}
G_{\beta_n,K_n}(x)=\frac{G_{\beta_n,K_n}^{(2)}(0)}{2}x^2+\frac{G_{\beta_n,K_n}^{(4)}(0)}{24}x^4+B_{\beta_n,K_n}(\xi(x))x^5.
\end{eqnarray}
The error terms $B_{\beta_n,K_n}(\xi(x))$ are uniformly bounded over $n\in\N$ and $x\in(-R,R)$. 
We have $G_{\beta_n,K_n}^{(4)}\left(0\right)= \frac{2(2\beta_nK_n)^4(4-e^{\beta_n})}{(e^{\beta_n}+2)^2}$ and $G_{\beta_n,K_n}^{(2)}\left(0\right) \to 0$
for $n \to \infty$.
\item For $(\beta_n,K_n)\rightarrow (\beta_c,K_c(\beta_c))\in C$ the type of the minimum point $0 \in {\mathcal M}_{\beta_c, K_c(\beta_c)}$ is $r=3$ and there exists $\xi=\xi(x)\in [-x,x]$ such that
\begin{align}\label{taylorGBEGC}
G_{\beta_n,K_n}(x)=\frac{G_{\beta_n,K_n}^{(2)}(0)}{2}x^2+\frac{G_{\beta_n,K_n}^{(4)}(0)}{24}x^4+\frac{G_{\beta_n,K_n}^{(6)}(0)}{6!}x^6+C_{\beta_n,K_n}(\xi(x))x^7.
\end{align}
The error terms $C_{\beta_n,K_n}(\xi(x))$ are uniformly bounded over $n\in\N$ and $x\in(-R,R)$. We have
$G_{\beta_n,K_n}^{(2)}\left(0\right) \to 0$ and $G_{\beta_n,K_n}^{(4)}\left(0\right) \to 0$ for $n \to \infty$.
\item Furthermore let us choose the sequence $K_n$ as in  
\eqref{KnBEG}, then in (2)-(3) we obtain
\begin{align}\label{2ABlGBEG}
G_{\beta_n,K_n}^{(2)}\left(0\right) = \frac{k}{n^{\Delta_2}} \, C_n^{(2)}
\end{align}
with $C_n^{(2)} = \frac{2 \beta_n K_n}{K_c(\beta_n)} \to 2 \beta_c$.
Moreover if we assume that $\beta_n$ is chosen as in \eqref{BetanBEG}  we have in (2)-(3) that
\begin{align}\label{4ABlGBEG}
G_{\beta_n,K_n}^{(4)}\left(0\right)=
\frac{b}{n^{\Delta_1}}\, C_n^{(4)}
\end{align}
with $C_n^{(4)} = \frac{2(2 \beta_n K_n)^4}{(e^{\beta_n}+2)^2} \to \frac 92$, see \cite[(4.9),(4.10),(7.4) and (7.5)]{CosteniucEllis:2007}. 
\end{enumerate}
\end{theorem}

We next preview the contents of the present paper. In the next section, Section 2, we state the limit theorems and the rates of convergence for the total spin per site. We are able to obtain 21 different limiting densities that result from the values of $k$ and $b$ defining the physically dissimilar regions of points, $A$, $B$ and $C$, the sequence $(\beta_n,K_n)$ is converging from. Formulating the corresponding rates of convergence we will see that the 21 cases split into 42 cases that result from the values of $\gamma$, $\Delta_1$ and $\Delta_2$ defined in \eqref{BetanBEG} and \eqref{KnBEG} and thus depend on the speed at which $(\beta_n, K_n) \to (\beta,K)$. The proofs of our Theorems will be presented in Section \ref{Subsec:Proofs}. They apply Stein's method, which shortly will be introduced in Section 3.

\section{Limit Theorems and rates of convergence}\label{subsec:Statement of Results BME}

We are prepared to state our results.
Because of the intricate structure of the model we are able to find three different limit theorems if $(\beta,K)$ is assumed to be fixed and 18 scaling limits for the total spin as $(\beta,K)$ converges along the sequences $(\beta_n,K_n)$ defined in \eqref{BetanBEG} and \eqref{KnBEG} to points belonging to the sets $A$, $B$ and $C$ defined in Section 1. 
Let \begin{align}\label{Wg}
W_{\gamma}:=\frac{S_n}{n^{1-\gamma}}\text{ with }\gamma\in(0,1/2].
\end{align}

\subsection{Rates of convergence for fixed $(\beta,K)$ and $(\beta_n,K_n) \to (\beta,K) \in A$}

First of all we assume that $(\beta,K)$ is fixed. Limit theorems for the spin per site were first discussed in \cite{CosteniucEllis:2007}.

\begin{theorem}\label{Tfixed}
Let $W_{\gamma}$ be defined in \eqref{Wg}. 
\begin{enumerate}
\item For $(\beta,K)\in A$ we have $\gamma=1/2$. If  $Z_A$ is a random variable distributed according to the normal distribution $N(0, \E W_{1/2}^2)$ on $\R$ with expectation zero and variance $\E (W_{1/2}^2)$, we have that
\begin{align} \label{r1}
\sup\limits_{t\in\R}\big| P\bigl(W_{1/2}\leq t\bigr) - P\bigl(Z_A\leq t\bigr) \big| \leq L_1 \cdot n^{-1/2},
\end{align}
for some constant $L_1$ depending only on $(\beta,K)$.
\item For $(\beta,K_c(\beta))\in B$ we have $\gamma=1/4$. If  $Z_B$ is a random variable distributed according to the probability measure on $\R$ with density
$
f_{B}(t):= L_2 \cdot \exp \left( -c t^4\right),
$
with $c = c(W_{1/4}) = (4 \E(W_{1/4}^4))^{-1}$,
and $L_2$ the appropriately chosen normalisation constant,
we have that
\begin{align*}
\sup\limits_{t\in\R}\big| P\bigl(W_{1/4}\leq t\bigr) - P\bigl(Z_B\leq t\bigr) \big| \leq L_3 \cdot n^{-1/4},
\end{align*}
for some constant $L_3$ depending only on $(\beta,K_c(\beta))$.
\item For $(\beta_c,K_c(\beta_c))$ we have $\gamma=1/6$. If  $Z_C$ is a random variable distributed according to the probability measure on $\R$ with density
$f_{C}(t):= L_4 \cdot \exp \left( - d t^6\right)$ with $d= d(W_{1/6})= (6 \E(W_{1/6}^6))^{-1}$ 
and $L_4$ the appropriately chosen normalisation constant,
we have that
\begin{align*}
\sup\limits_{t\in\R}\big| P\bigl(W_{1/6}\leq t\bigr) - P\bigl(Z_C\leq t\bigr) \big| \leq L_5 \cdot n^{-1/6},
\end{align*}
for some constant $L_5$ depending only on $(\beta_c,K_c(\beta_c))$.
\end{enumerate}
\end{theorem}

\begin{remark}
Theorem \ref{Tfixed} shows that the rate of convergence is affected by the set containing $(\beta,K)$. For the BEG model, to the best of our knowledge, our 
results are the first ones, where the quality of approximation was estimated. In region $A$ we found an optimal rate $n^{-1/2}$ for
the Kolmogorov distance known as a Berry-Esseen type result. We do not know whether the other rates are optimal. There is one case known in the
literature, where the limiting density is of the form $\exp(- {\rm const.} x^4)$ and the rate of convergence is of order $n^{-1/2}$: This is the rescaled total spin in the classical Curie-Weiss model at the critical temperature $\beta_c$, see \cite[Theorem 3.8]{Eichelsbacher/Loewe:2010}. The technical reason is that
the Taylor expansion of the corresponding function $G_{\beta}$ is given by the Taylor expansion of $\tanh(\cdot)$. In the same paper, generalisations of the Curie-Weiss model lead to bounds of order $n^{-1/(2k)}$ whenever the limiting density is of type  $\exp(- {\rm const.} x^{2k})$.
Note that Theorems 5.5, 6.1 (Case 1) and Theorem 7.1 (Case 1)
in \cite{CosteniucEllis:2007} follow from our result. We present a proof by Stein's method, avoiding the application of transforms. The limiting densities in \cite{CosteniucEllis:2007} show moreover, that
for $(\beta,K)\in A$ we have $\lim_{n \to \infty} \E(W_{1/2}^2) = \bigl( G_{\beta, K}^{(2)}(0) \bigr)^{-1} - (2 \beta K)^{-1} $, for $(\beta,K)\in B$ we have $\lim_{n \to \infty} 4 \E(W_{1/4}^4) = \frac{2^3 4}{(e^{\beta} +2)^2(4-e^{\beta})}$ (see \cite[(6.5)]{CosteniucEllis:2007}). Finally for $(\beta_c,K_c(\beta_c))$
we have $\lim_{n \to \infty} 6 \E(W_{1/6}^6) = \frac{40}{9}$ (see \cite[Theorem 7.1]{CosteniucEllis:2007}). The choices of our densities of the random variables $Z_A, Z_B$ and $Z_C$ will be explained in Section 3 in more detail. If the limiting density of a random variable is not known and especially the limiting
moments are unknown, it is a remarkable advantage of applying Stein's method to be able to compare the distribution of a random variable $W$ of interest with
a distribution which inherits some moments of $W$, which characterise the limiting distribution.
\end{remark}

If $(\beta_n,K_n)$ denotes a positive, bounded sequence converging to $(\beta, K)\in A$ the situation of the scaling limits is as follows:

\begin{theorem}\label{nA}
Let $(\beta_n,K_n)$ be an arbitrary positive, bounded sequence that converges to $(\beta,K)\in A$. Then we obtain $\gamma =1/2$ 
and the same result as in \eqref{r1}, Theorem \ref{Tfixed}.
\end{theorem}

\subsection{Six rates of convergence for $(\beta_n,K_n) \to (\beta, K_c(\beta)) \in B$}

If $(\beta_n,K_n)$ denotes a positive, bounded sequence converging to $(\beta, K_c(\beta))\in B$ the situation of the scaling limits gets more complicated. The form of the limit theorem depends on the Taylor expansion of $G_{\beta_n,K_n}$ in the neighbourhood of the global minimum point 0. This will become evident in the proof of the next theorem and is physically motivated in \cite[Section 6]{CosteniucEllis:2007}. 

\begin{theorem}\label{nB}
For fixed $\beta\in(0,\beta_c)$, let $\beta_n$ be an arbitrary positive, bounded sequence that converges to $\beta$ and $K_n$ be the sequence defined in \eqref{KnBEG}. Let $W_{\gamma}$ be defined in \eqref{Wg}. Then by continuity of $K_c(\cdot)$ we have $(\beta_n,K_n)\rightarrow (\beta, K_c(\beta))\in B$. 
Assume that 
\begin{align*}
v=\min(2\gamma+\Delta_2-1, 4\gamma-1)=0.
\end{align*}
and let $\delta(a,b)$ equal 1 if $a=b$ and equal 0 if $a\neq b$.
Let $Z_{a_1, a_2}$ be a random variable distributed according to a densities of the form
\begin{align}\label{fnB}
f_{a_1,a_2}(x):= C \cdot \exp \left( -(\delta(v, 2\gamma+\Delta_2-1) a_1 x^2+ \delta(v, 4\gamma-1) a_2 x^4)\right),
\end{align}
for certain constants $a_1$, $a_2$ and $C$.
\begin{enumerate}
\item If $\gamma = 1/4$ and $\Delta_2 = 1/2$ there exist explicit constants $a_1, a_2 \not=0$ (depending on $\beta_n$, $K_n$, $\E W_{1/4}^2$ and $\E W_{1/4}^4$ presented explicitly in the proof) such that for a constant $C$
\begin{align*}
\sup\limits_{t\in\R}\big| P\bigl(W_{1/4}\leq t\bigr) - P\bigl(Z_{a_1, a_2} \leq t\bigr) \big| \leq C n^{-1/4}.
\end{align*}
\item If $2 \gamma = 1 - \Delta_2$, $\gamma \in (1/4, 1/2)$, $\Delta_2 \in (0, 1/2)$, we take $a_1 = (2 \E(W_{\gamma}^2))^{-1}>0$ and $a_2=0$ and obtain
with a constant $C$
\begin{align*}
\sup\limits_{t\in\R}\big| P\bigl(W_{\gamma}\leq t\bigr) - P\bigl(Z_{a_1, a_2} \leq t\bigr) \big| \leq C \begin{cases} n^{1 - 4 \gamma} &, \gamma \in (1/4, 1/3],
\\ n^{-\gamma} &, \gamma \in [1/3, 1/2). \end{cases}
\end{align*}
\item If $\gamma = 1/4$ and $\Delta_2 > 1/2$ we take $a_1 = 0$ and $a_2 = (4 \E(W_{1/4}^4))^{-1}$ and obtain with a constant $C$
\begin{align*}
\sup\limits_{t\in\R}\big| P\bigl(W_{1/4}\leq t\bigr) - P\bigl(Z_{a_1, a_2} \leq t\bigr) \big| \leq C \begin{cases} n^{-(\Delta_2-1/2)} &, \Delta_2 \in (1/2, 3/4),
\\ n^{-1/4} &, \Delta_2 \geq 3/4. \end{cases}
\end{align*}
\end{enumerate}
\end{theorem}

\begin{remark}\label{v}
We observe that the limit theorems depend on the value of $\gamma$ and on $K_n$ through the speed $\Delta_2$. 
In the first case (1) the limit-density is $\exp(-a_1 x^2 - a_2 x^4)$ and hence the case is influenced by regions A and B and the rate is $n^{- \frac 14}$.
The case corresponds to the critical speed $\Delta_2 = 1/2$. The coefficient $a_1$ depends on the sign of $k \not=0$ and hence yields two different limit densities, whereas $a_2 >0$ (both can be seen from the proof). 

In the second case, the limit-density is $\exp(- a_1 x^2)$ with $a_1>0$. With $\Delta_2 \in (0, 1/2)$ it corresponds to a slow convergence of $K_n$
to $K_c(\beta)$. In this case only region $A$ influences the form of the limiting density. But we consider converging in distribution to a normal distribution even though the non-classical scaling is given by $n^{1 - \gamma}$ with $\gamma \in (1/4, 1/2)$. Now we consider an {\it additional phase transition phenomenon}, since the rate of convergence depends on $\gamma$: the breakpoint is $\gamma=1/3$ and the more classical the scaling the better the rate of convergence. In other words we see that for $\gamma \in (1/4,1/3]$ the influence of region $A$ is getting weaker and weaker in the sense of a slower
rate of convergence. 
  
In the third case the limit-density is $\exp(-a_2 x^4)$ for any $k \in \N$.  This case is linked to an influence of $B$. Here the speed $\Delta_2$
is most rapid. Anyhow, an {\it additional phase transition phenomenon} occurs: the rate of convergence depends on $\Delta_2$ and is getting best for $\Delta_2 \geq 3/4$. So if we choose a speed $\Delta_2 \geq 3/4$ we obtain the best rate $n^{-1/4}$ or in other words: we can force the speed to obtain a rate of convergence which is optimal in comparison to the observation in Theorem \ref{Tfixed}. For any $\Delta_2 \in (1/2, 3/4)$ we have the rate $n^{-(\Delta_2-1/2)}$.

Summarising we consider 6 different cases with 4 different limit densities (compare with Figure 4 on page 530 in \cite{CosteniucEllis:2007}). In 5 cases (with the help of a certain speed up of $K_n$) we obtain the same rates of convergence as for fixed $(\beta,K)$, see Theorem \ref{Tfixed}. A phase transition phenomena persists in the case of a non-classical scaling. 
\end{remark}

\begin{remark}
As shown in \cite[Theorem 6.1]{CosteniucEllis:2007} we focus only on the case of $v=0$. If $v>0$ one is not able to obtain any limit theorem. For $v<0$ the authors in \cite{CosteniucEllis:2007} obtain moderate deviation principles for the total spin per site. The fact that $v$ is required to be zero becomes evident in Lemma \ref{MomBEG}. 
\end{remark}

\subsection{Thirtytwo rates of convergence for $(\beta_n,K_n) \to (\beta_c, K_c(\beta_c))$}

The last limit theorem of this section correspond to the case that the sequence $(\beta_n,K_n)$ converges to the tricritical point $(\beta_c,K_c(\beta_c))=(\log 4,3/[2\log 4])$. Here we also need the sequence $\beta_n\rightarrow \beta$ taken from \eqref{BetanBEG}.

\begin{theorem}\label{nC}

Let $\beta_n$ and $K_n$ be the sequences defined in \eqref{BetanBEG} and \eqref{KnBEG}. 
Then $(\beta_n,K_n)\rightarrow(\beta_c,K_c(\beta_c))$. Let $W_{\gamma}$ be defined in \eqref{Wg}. Given $\gamma\in [1/6,1/2]$ we assume that 
$$
w=\min(2\gamma+\Delta_2-1, 4\gamma+\Delta_1-1, 6\gamma-1)=0.
$$ 
Then, if the random variable $Z_{b_1, b_2, b_3}$ is distributed according to the probability measure on $\R$ with the density
\begin{align}\label{fnC}
f_{b_1,b_2,b_3}(x):= C \cdot \exp \left( - \left( \delta(w, 2\gamma+\Delta_2-1) b_1 x^2+\delta(w, 4\gamma+\Delta_1-1)b_2 x^4+\delta(w, 6\gamma-1) b_3 x^6
\right) \right)
\end{align}
for certain constants $b_1, b_2, b_3$ and $C$.

\begin{enumerate}
\item If $\gamma = 1/6$ and $\Delta_1 = 1/3$, $\Delta_2 = 2/3$ there exists explicit constants $b_1, b_2, b_3 \not=0$ (depending on $\beta_n$, $K_n$, $\E W_{1/6}^i$ with $i \in \{2,4,6\}$ presented explicitly in the proof) such that for a constant $C$
\begin{align*}
\sup\limits_{t\in\R}\big| P\bigl(W_{1/6}\leq t\bigr) - P\bigl(Z_{b_1, b_2, b_3} \leq t\bigr) \big| \leq C n^{-1/6}.
\end{align*}
\item If $2 \gamma = 1 - \Delta_2$, $\gamma \in (1/4, 1/2)$, $\Delta_2 \in (0, 1/2)$ and $\Delta_1 >0$ we take $b_1 = (2 \E(W_{\gamma}^2))^{-1}>0$ and $b_2=b_3=0$ and obtain for a constant $C$
\begin{align*}
\sup\limits_{t\in\R}\big| P\bigl(W_{\gamma}\leq t\bigr) - P\bigl(Z_{b_1, b_2, b_3} \leq t\bigr) \big| \leq C \begin{cases} n^{1 - 4 \gamma-\Delta_1} &, \gamma \in (1/4, 1/3], \Delta_1 \in (0, 1 - 3 \gamma), \\ n^{-\gamma} &, \gamma \in (1/4,1/3],
\Delta_1 \geq 1 - 3 \gamma, 
\\ n^{-\gamma} &, \gamma \in [1/3, 1/2). \end{cases}
\end{align*}
\item If $2 \gamma = 1 - \Delta_2$, $\gamma \in (1/6, 1/4]$, $\Delta_2 \in [1/2, 2/3)$ and $\Delta_1 >2 \Delta_2 -1$ we take $b_1 = (2 \E(W_{\gamma}^2))^{-1}>0$ and $b_2=b_3=0$ and obtain for a constant $C$
\begin{align*}
\sup\limits_{t\in\R}\big| P\bigl(W_{\gamma}\leq t\bigr) - P\bigl(Z_{b_1, b_2, b_3} \leq t\bigr) \big| \leq C \begin{cases} n^{1 - 4 \gamma-\Delta_1} &, \gamma \in (1/6, 1/5], 1- 4 \gamma < \Delta_1 <  2 \gamma, \\ n^{1- 6 \gamma} &, \gamma \in (1/6,1/5], \Delta_1 \geq 2 \gamma, \\
n^{1 - 4 \gamma-\Delta_1} &, \gamma \in [1/5, 1/4], 1- 4 \gamma < \Delta_1 <  1 - 3 \gamma, \\
n^{-\gamma} &, \gamma \in [1/5, 1/4], \Delta_1 \geq 1 - 3 \gamma. \end{cases}
\end{align*}
\item If $\gamma = 1/6$, $\Delta_1 > 1/3$ and $\Delta_2 > 2/3$ we take $b_1 =b_2= 0$ and $b_3 = (6 \E(W_{1/6}^6))^{-1}>0$ and obtain for a constant $C$
\begin{align*}
\sup\limits_{t\in\R}\big| P\bigl(W_{1/6}\leq t\bigr) - P\bigl(Z_{b_1, b_2, b_3} \leq t\bigr) \big| \leq C \begin{cases} n^{1/3 - \Delta_1} &, \Delta_1 \in (1/3,1/2), \Delta_2 \in (2/3, 5/6), \\ & \,\, \,\Delta_1 \leq \Delta_2 - 1/3, \\ n^{2/3 - \Delta_2} &, \Delta_1 \in (1/3,1/2), \Delta_2 \in (2/3, 5/6), \\ & \,\,\,\Delta_1 > \Delta_2 - 1/3,\\ 
n^{1/3 - \Delta_1} &, \Delta_1 \in (1/3, 1/2), \Delta_2 \geq 5/6, \\
n^{2/3 - \Delta_2} &, \Delta_1 \geq 1/2, \Delta_2 \in (2/3, 5/6), \\
n^{-1/6} &, \Delta_1 \geq 1/2, \Delta_2 \geq 5/6. \end{cases}
\end{align*}
\item If $4 \gamma = 1 - \Delta_1$, $\gamma \in (1/6,1/4)$, $\Delta_1 \in (0, 1/3)$ and $2 \Delta_2 > \Delta_1 +1$ we take $b_1 =b_3= 0$ and $b_2 = (4 \E(W_{\gamma}^4))^{-1}>0$ and obtain for a constant $C$
\begin{align*}
\sup\limits_{t\in\R}\big| P\bigl(W_{\gamma}\leq t\bigr) - P\bigl(Z_{b_1, b_2, b_3} \leq t\bigr) \big| \leq C \begin{cases} n^{1 - 2 \gamma-\Delta_2} &, \gamma \in (1/6, 1/5), \Delta_2 < 4 \gamma, \\ n^{1- 6 \gamma} &, \gamma \in (1/6,1/5), \Delta_2 \geq 4 \gamma, \\
n^{1 - 2 \gamma-\Delta_2} &, \gamma \in [1/5, 1/4), \Delta_2 <  1 - \gamma, \\
n^{-\gamma} &, \gamma \in [1/5, 1/4), \Delta_2 \geq 1 - \gamma. \end{cases}
\end{align*}
\item If $\gamma = 1/6$, $\Delta_1 =1/3$ and $\Delta_2 > 2/3$ we take $b_1=0$ and $b_2, b_3 \not= 0$  (depending on $\beta_n$, $K_n$, $\E W_{1/6}^i$ with $i \in \{4,6\}$ presented explicitly in the proof) and obtain for a constant $C$
\begin{align*}
\sup\limits_{t\in\R}\big| P\bigl(W_{1/6}\leq t\bigr) - P\bigl(Z_{b_1, b_2, b_3} \leq t\bigr) \big| \leq C \begin{cases} n^{-(\Delta_2-2/3)} &, \Delta_2 \in (2/3, 5/6),
\\ n^{-1/6} &, \Delta_2 \geq 5/6. \end{cases}
\end{align*}
\item If $\gamma = 1/6$, $\Delta_1 >1/3$ and $\Delta_2 = 2/3$ we take $b_2=0$ and $b_1, b_3 \not= 0$  (depending on $\beta_n$, $K_n$, $\E W_{1/6}^i$ with $i \in \{2,6\}$ presented explicitly in the proof) and obtain for a constant $C$
\begin{align*}
\sup\limits_{t\in\R}\big| P\bigl(W_{1/6}\leq t\bigr) - P\bigl(Z_{b_1, b_2, b_3} \leq t\bigr) \big| \leq C \begin{cases} n^{-(\Delta_1-1/3)} &, \Delta_1 \in (1/3, 1/2),
\\ n^{-1/6} &, \Delta_1 \geq 1/2. \end{cases}
\end{align*}
\item If $4 \gamma = 1 - \Delta_1$, $\gamma \in (1/6, 1/4)$, $\Delta_1 \in(0,1/3)$ and $2 \Delta_2 = \Delta_1 +1$ we take $b_3=0$ and $b_1, b_2 \not= 0$  (depending on $\beta_n$, $K_n$, $\E W_{1/6}^i$ with $i \in \{2,4\}$ presented explicitly in the proof) and obtain for a constant $C$
\begin{align*}
\sup\limits_{t\in\R}\big| P\bigl(W_{\gamma}\leq t\bigr) - P\bigl(Z_{b_1, b_2, b_3} \leq t\bigr) \big| \leq C 
\begin{cases} n^{1 - 6 \gamma} &, \gamma \in (1/6, 1/5],
\\ n^{-\gamma} &, \gamma \in [1/5, 1/4). \end{cases}
\end{align*}
\end{enumerate}
\end{theorem}

\begin{remark}
We observe that the limit theorems depend on the value of $\gamma$ and on $\beta_n$ and $K_n$ through the speeds $\Delta_1$, $\Delta_2$. 
In the first case (1) the limit-density is $\exp(-b_1 x^2 - b_2 x^4 -b_3 x^6)$ and hence the case is influenced by regions A, B and C and the rate is $n^{- \frac 16}$. The case corresponds to the critical speeds $\Delta_1 = 1/3$ and $\Delta_2 = 2/3$. The coefficient $b_1$ depends on the sign of $k \not=0$, the
coefficient $b_2$ on the sign of $b \not= 0$ and hence yields 4 different limit densities, whereas $b_3 >0$ (both can be seen from the proof). 

The second case should be compared with the second case of Theorem \ref{nB}: in addition to the conditions in Theorem \ref{nB} (2), we assume that $\beta_n$ converges
with speed $\Delta_1$ to $\beta_c$.  Again we consider convergence in distribution to a normal distribution even though the non-classical scaling is given by $n^{1 - \gamma}$ with $\gamma \in (1/4, 1/2)$ and again the speed is $n^{-\gamma}$ for any $\gamma \in [1/3, 1/2)$, independent of $\Delta_1$.
But for $\gamma < 1/3$ the rate of convergence is of order $n^{1-4 \gamma - \Delta_1}$ and hence slower than in region $B$. But if we speed up
$\beta_n$ in choosing $\Delta_1 \geq 1 - 3 \gamma$ the rate of convergence is $n^{-\gamma}$ and hence faster than in region $B$. In total there are 3 subcases.

The third case reads as follows. Under the same relation $2 \gamma = 1 - \Delta_2$ as in case (2), with a speed up of $\Delta_1$ and $\Delta_2$ it
is possible to observe normal convergence even for the scaling $\gamma \in (1/6, 1/4]$. For $\gamma \geq 1/5$ it is possible to obtain the rate $n^{-\gamma}$
if we speed up $\beta_n$, for $\gamma \leq 1/5$ a speed up of $\beta_n$ implies the rate $n^{1-6 \gamma}$, which could have been expected in comparison to Theorem \ref{nB} (2). Case (2) and (3) are linked to an influence of $A$. In total we have 4 subcases.

The fourth and fifth case are linked to the limiting densities $\exp(-b_3 x^6)$ (influenced only by $C$) and $\exp(-b_2 x^4)$ (influenced only by region $B$), respectively.
In case (4) a speed up of both $K_n$ and $\beta_n$ leads to the rate $n^{-1/6}$, in case (5) the result is comparable with case (3): a certain
speed up leads to rate $n^{-\gamma}$ or $n^{1 - 6 \gamma}$ depending on $\gamma$. Interesting enough one obtains converging in distribution 
to $\exp(-b_2 x^4)$ even though the non-classical scaling is given by $n^{1 - \gamma}$ with $\gamma \in (1/6, 1/4)$, which is comparable with case
(2) and (3). Summarising we have 9 subcases.

The last three cases (6), (7) and (8) are linked to limiting densities $\exp(-b_2 x^4 -b_3 x^6)$ (influenced by $B$ and $C$), $\exp(-b_1 x^2 -b_3 x^6)$ (influenced by region $A$ and $C$) and $\exp(-b_1 x^2 -b_2 x^4)$ (influenced by $A$ and $B$). Cases (6) and (7) are comparable with case (3) in Theorem \ref{nB}: a certain speed up of $\Delta_1$ and $\Delta_2$, respectively, leads to the rate $n^{-1/6}$. Finally case (8) is comparable with case (2) in Theorem \ref{nB} with a non-classical scaling. In all three cases one of the nonzero parameters $b_i$ depends on the sign of $k \not=0$ and $b \not =0$, respectively, and hence yields two different limit densities. Hence we have 4 cases each. 
\end{remark}

\begin{remark}
As shown in \cite[Theorem 7.1]{CosteniucEllis:2007} we focus only on the case of $w=0$. If $w>0$ one is not able to obtain any limit theorem. For $w<0$ the authors in \cite{CosteniucEllis:2007} obtain moderate deviation principles for the total spin per site. The fact that $w$ is required to be zero becomes important in Lemma \ref{MomBEG}. 
\end{remark}

Summarising we consider 32 different cases with 13 different limit densities (compare with Table IV. on page 536 in \cite{CosteniucEllis:2007}). 
In all but three cases (with the help of a certain speed up of $\beta_n$ and $K_n$, respectively) we obtain the same rates of convergence as for fixed $(\beta,K)$, see Theorem \ref{Tfixed}. A phase transition phenomena persists in 
three cases of non-classical scalings: cases (3), (5) and (8).

Thus, combining the theorems above, we have 42 limit theorems depending on the values of $(\beta,K)$. Their proofs can be found in Section 4.

\newpage
\section{Stein's method}

Stein introduced in \cite{Stein:1986} the exchangeable pair approach. Given a random variable $W$, Stein's method
is based on the construction of another variable $W'$ (some coupling) such that the pair $(W, W')$ is exchangeable, i.e. their
joint distribution is symmetric.  A theorem of Stein (\cite[Theorem 1, Lecture III]{Stein:1986})
shows that a measure of proximity of $W$ to normality may be provided in terms
of the exchangeable pair, requiring $W'-W$ to be sufficiently small. He assumed the condition
$$
\E(W'|W)=(1-\lambda) \, W
$$
for some $0 < \lambda <1$. 
Heuristically, this condition can be understood as a linear regression condition: if $(W,W')$ were bivariate normal with correlation $\varrho$, then
$\E [W' |W] = \varrho \, W$ and the condition would be satisfied with $\lambda = 1 - \varrho$.
Stein's approach has been successfully applied in many models, see e.g. \cite{Stein:1986} or
\cite{DiaconisStein:2004} and references therein. In \cite{Rinott/Rotar:1997}, the range of application was
extended by replacing the linear regression property by a weaker condition.
We consider Stein's method by replacing the linear regression property by
\begin{equation} \label{exchangepsi}
\E(W'|W) = W + \lambda \, \psi(W) - R(W),
\end{equation}
where $\psi(x)$ depends on a continuous distribution under consideration and $R(W)$ is a remainder term.
Recently in \cite{Eichelsbacher/Loewe:2010} and \cite{Doebler:2012} the exchangeable pair approach was extended to more absolutely continuous univariate distributions with a nice collection of new applications.

Given two random variables $X$ and $Y$ defined on a common probability space, we denote the Kolmogorov distance of the distributions of $X$ and $Y$ by
$$
d_{\rm{K}}(X,Y) := \sup_{z \in \R} | P(X \leq z) - P(Y \leq z)|.
$$
Let $I=(a,b)$ be a real interval, where $-\infty \leq a < b \leq \infty$. A function is called {\it regular}
if $f$ is finite on $I$ and, at any interior point of $I$, $f$ possesses a right-hand limit and a left-hand limit. Further, $f$ possesses
a right-hand limit $f(a+)$ at the point $a$ and a left-hand limit $f(b-)$ at the point $b$.
Let us assume, that the regular density $p$ satisfies the following condition:

{\bf Assumption (D)}
Let $p$ be a regular, strictly positive density on an interval $I=[a,b]$. Suppose $p$ has a derivative $p'$ that is
regular on $I$, has only countably many sign changes, and is continuous at the sign changes. Suppose moreover
that $\int_I p(x) | \log(p(x))| \, dx < \infty$
and that
$\psi(x) := \frac{p'(x)}{p(x)}$ is regular.

In \cite{DiaconisStein:2004} it is proved, that a random variable $Z$ is distributed according to the density $p$ if and only if
$
\E \bigl( f'(Z) + \psi(Z) \, f(Z) \bigr) = f(b-) \, p(b-) - f(a+) \, p(a+)
$
for a suitably chosen class $\mathcal{F}$ of functions $f$. 
The corresponding Stein identity is
\begin{equation} \label{steinid2}
f'(x) + \psi(x) \, f(x) = h(x) - P(h),
\end{equation}
where $h$ is a measurable function for which $\int_I |h(x)|\, p(x) \, dx < \infty$, $P(x) := \int_{-\infty}^x p(y) \, dy$
and $P(h) := \int_I h(y) \, p(y)\, dy$.
The solution $f:=f_h$ of this differential equation is given by
\begin{equation} \label{solution}
f(x) = \frac{ \int_a^x \bigl( h(y) - Ph \bigr) \, p(y) \, dy}{p(x)}.
\end{equation}
For the function $h(x) := 1_{\{x \leq z\}}(x)$ let $f_z$ be the corresponding solution of \eqref{steinid2}.

{\bf Assumption (B)}
Let $p$ be a density fulfilling Assumption (D) 
We assume that the solution $f_z$ of
$
f_z'(x) + \psi(x) \, f_z(x) = 1_{\{x \leq z\}}(x) - P(z)
$
satisfies
$$
|f_z(x)| \leq d_1, \quad |f_z'(x)| \leq d_2 \quad \text{and} \quad
|f_z'(x)-f_z'(y) | \leq d_3
$$
and
\begin{equation} \label{addcond}
|(\psi(x) \, f_z(x))'| = \bigl| ( \frac{p'(x)}{p(x)} \, f_z(x))' \bigr| \leq d_4
\end{equation}
for all real $x$ and $y$, where $d_1, d_2, d_3$ and $d_4$ are constants.

We will apply the following results proved in \cite{Eichelsbacher/Loewe:2010}.
Let $p_W$ be a probability density such that a random variable $Z$ is distributed according to $p_W$ if and only
if $
\E \bigl( \E[W \psi(W)] \, f'(Z) + \psi(Z) \, f(Z) \bigr) =0
$
for a suitably chosen class of functions.

\begin{theorem} (see Theorem 2.5 in \cite{Eichelsbacher/Loewe:2010}) \label{generaldensity2}
Let $p$ be a density fulfilling Assumption (D).
Let $(W, W')$ be an exchangeable pair of random variables such that \eqref{exchangepsi} holds with respect to $p$ ($\psi= p'/p$).
If $Z_W$ is a random variable distributed according to $p_W$, we assume that 
the solutions $f_z$ of $ \E[W \psi(W)] \, f'(x) + \psi(x) \, f(x)  = 1_{\{x \leq z\}}(x) -P(z)$ fulfill Assumption (B). Then
for any $A >0$ one has
\begin{eqnarray} \label{kolall2}
d_{\rm{K}}(W, Z_W) & \leq & \frac{d_2}{2 \lambda} \bigl( {\rm Var} \bigl( \E [ (W-W')^2| W] \bigr)^{1/2} +
\big( d_1 + d_2 \sqrt{\E(W^2)} +\frac 32 A \bigr)  \frac{\sqrt{\E(R^2)}}{\lambda} \nonumber \\
& + &  \frac{1}{\lambda} \bigl( \frac{d_4 A^3}{4} \bigr) + \frac{3A}{2} \E (|\psi(W)|)
 +   \frac{d_3}{2 \lambda} \E \bigl( (W-W')^2 1_{\{|W-W'| \geq A\}} \bigr).
\end{eqnarray}
\end{theorem}

\begin{remark}
In case the regression property \eqref{exchangepsi} is fulfilled with $\psi=p'/p$, we expect a comparison of the distribution of $W$ with $Z$ distributed according to the regular Lebesgue-density $p$. Why do we introduce the modified density $p_W$? The reason was already discussed in \cite{Eichelsbacher/Loewe:2010}:
If \eqref{exchangepsi} is fulfilled, on obtains that $\E(W - W')^2 = - 2 \lambda \E[W \psi(W)] + 2 \E [ W R(W) ]$. Comparing the distribution $W$ with
$Z$ distributed according to $p$ leads to a plug-in theorem (see \cite[Theorem 2.4]{Eichelsbacher/Loewe:2010}), where one has to estimate a term like
$$
\E \biggl| 1 - \frac{1}{2 \lambda} \E [ (W' -W)^2 | W] \biggr|.
$$
But with our observation $ \E \bigl(  1 - \frac{1}{2 \lambda} \E [ (W' -W)^2 | W] \bigr)= 1 + \E[W \psi(W)] - \frac{1}{\lambda} \E [ W R(W) ]$.
Therefore the bounds in Theorem 2.4 in 
\cite{Eichelsbacher/Loewe:2010} are only useful, if $- \E[W \psi(W)] $ is close to 1 and  $\frac{1}{\lambda} \E [ W R(W) ]$ is small. Alternatively, bounds can be obtained by comparing with a modified distribution that involves  $\E [ W \psi(W) ]$. This leads to $p_W$. Note that this is compatible with the quite
general approach introduced in \cite{Doebler:2012}.
\end{remark}

In the following corollary, we discuss the Kolmogorov-distance of the distribution of a random variable $W$ to a
random variable distributed according to $N(0, \E(W^2))$.

\begin{cor} (see Corollary 2.10 in \cite{Eichelsbacher/Loewe:2010})\label{corsigma}
Let $\sigma^2 >0$ and  $(W, W')$ be an exchangeable pair of real-valued random variables such that
\begin{equation} \label{3.2}
E(W' |W) = \bigl(1 - \frac{\lambda}{\sigma^2} \bigr) W -R(W)
\end{equation}
for some random variable $R(W)$ and with $0 < \lambda <1$. Assume that $\E(W^2)$ is finite and $|W-W'| \leq A$ for a constant $A$.
Let $Z_{W}$ denote a random variable distributed according to $N(0, \E(W^2))$. We obtain
\begin{eqnarray} \label{mainbound3}
d_{\rm{K}}(W, Z_{W}) & \leq & \frac{\sigma^2}{2 \lambda}  \bigl( {\rm Var} \bigl( \E[(W'-W)^2 |W] \bigr) \bigr)^{1/2}  +
\sigma^2 \biggl( \frac{\sqrt{\E(W^2)}\,  (\sqrt{2 \pi}+4) }{4} + 1.5 A \biggr) \frac{\sqrt{\E(R^2)}}{\lambda} \nonumber \\
& & \hspace{-2cm} + \sigma^2 \, \frac{A^3}{\lambda} \biggl(\frac{\sqrt{\E(W^2)} \, \sqrt{2 \pi}}{16} + \frac{\sqrt{\E(W^2)}}{4} \biggr)
+  \sigma^2 \, 1.5 A \, \sqrt{\E(W^2)}. 
\end{eqnarray}
\end{cor}

\section{Proofs}\label{Subsec:Proofs}

While developing an exchangeable pair $(W_{\gamma},W_{\gamma}')$ and applying Stein's method for the BEG model we will be confronted with the conditional expectation of a single spin $\omega_i$ and of the product of two single spins $\omega_i \, \omega_j$. Before proving the theorems we will collect some auxiliary results that will be needed in the sequel.  The proofs will be quite elementary. Interesting enough the statements which follow will be the basis of our proofs.

\begin{lemma}\label{BedBEG}
Let $\omega_i\in\{-1,0,1\}$, $S_n$ defined in \eqref{totalspin} and $S_n^i:=S_n-\omega_i$. Then
$$
\E\left[\omega_i| (\omega_k)_{k \neq i} \right]=f_{\beta,K}(S_n^i/n) \bigl( 1 + {\mathcal O}(1/n) \bigr)
$$
with
\begin{align}\label{fBetaBEG}
f_{\beta,K}(x):=
\frac{ 2 e^{-\beta}  \sinh(2 \beta K x)}{1+ 2 e^{-\beta} \cosh(2\beta Kx)}.
\end{align}
\end{lemma}

\begin{proof}
First of all we calculate the conditional probability of a given single spin
\begin{eqnarray} \nonumber
P_{\beta,K,n}(\omega_i=t| (\omega_k)_{k\neq i})&=&\frac{P_{\beta,K,n}(\{\omega_i=t\}\cap\{ (\omega_k)_{k\neq i}\})}{P_{\beta,K,n}((\omega_k)_{k\neq i})}\\\nonumber
&=&\frac{\exp \bigl( -\beta t^2+\frac{\beta K}{n} \bigl( t^2+2t\sum\limits_{k\neq i}\omega_k \bigr) \bigr)}{\sum\limits_{l\in\{-1,0,1\}}\exp\bigl(-\beta l^2+\frac{\beta K}{n}\bigl(l^2+2l\sum\limits_{k\neq i}\omega_k\bigr)\bigr)}\\\label{condsing}
&=&\frac{\exp\bigl(-\beta t^2+\frac{\beta K}{n}\bigl( t^2+2tS_n^i\bigr)\bigr)}{\sum\limits_{l\in\{-1,0,1\}}\exp\bigl(-\beta l^2+\frac{\beta K}{n}\bigl(l^2+2lS_n^i\bigr)\bigr)}.
\end{eqnarray}
Thus we obtain
$$
\E\left[\sigma_i| (\sigma_k)_{k\neq i} \right]=
\frac{ e^{\frac{\beta K}{n}} \, 2 e^{-\beta} \sinh \bigl(2 \beta K \frac{S_n^i}{n} \bigr)}{\sum\limits_{l\in\{-1,0,1\}}\exp\bigl( -\beta l^2+\frac{\beta K}{n}\bigl(l^2+2lS_n^i\bigr) \bigr)}
$$
and  with $|t| \leq 1$ it follows $\E\left[\sigma_i| (\sigma_k)_{k\neq i} \right] \leq e^{2 \beta K /n} f_{\beta,K}(S_n^i/n)$ and $\E\left[\sigma_i| (\sigma_k)_{k\neq i} \right] \geq e^{-2 \beta K /n} f_{\beta,K}(S_n^i/n)$. Hence the result is proved.
\end{proof}

The next lemma will connect the function $f_{\beta,K}$ defined in \eqref{fBetaBEG} and the function $G_{\beta,K}$ taken from \eqref{GBEG}.
\begin{lemma}\label{GzuBedBEG}
With the notions of Lemma \ref{BedBEG} we have
\begin{eqnarray*}
f_{\beta,K}(S_n^i/n)&=& S_n^i/n - \frac{1}{2\beta K}G_{\beta,K}^{(1)}\left(S_n^i/n\right).
\end{eqnarray*}
\end{lemma}
\begin{proof}
A direct calculation and \eqref{CBEG} yields
\begin{eqnarray}\label{Gstrich}\nonumber
G_{\beta,K}^{(1)}(x)&=&\frac{\partial}{\partial x}\bigl( \beta Kx^2-c_{\beta}(2\beta Kx)\bigr)\\\nonumber
&=&2\beta K\biggl(x-\frac{1+2e^{-\beta}}{1+e^{-\beta}\bigl(e^{2\beta Kx}+e^{-2\beta Kx}\bigr)}\cdot \frac{e^{-\beta}\bigl(e^{2\beta Kx}-e^{-2\beta Kx}\bigr)}{1+2e^{-\beta}}\biggr)\\
&=&2\beta K\biggl( x-\frac{ 2 e^{-\beta}\sinh(2\beta Kx)}{1+2 e^{-\beta} \cosh(2\beta Kx)}\biggr).
\end{eqnarray}
which yields the result.
\end{proof}

In order to get a bound on some variances we investigate the covariances for $i\neq j$.

\begin{lemma}\label{doppelspin}
For $i\neq j$, $i,j\in\{1,\ldots,n\}$, we have
\begin{align*}
\E[\omega_i^2\omega_j^2|(\omega_l)_{l\notin\{i,j\}}]&=f_1(S_n^{i,j}/n) \bigl( 1 + {\mathcal O} (1/n) \bigr)
\end{align*}
with
$$
f_1(x):=f_{1, \beta, K}(x) := \frac{2e^{-2\beta} (1 + \cosh(4\beta K x))}{1+4e^{-\beta}\cosh(2\beta K x)+2e^{-2\beta} (1+\cosh(4\beta K x))}
$$
and $S_n^{i,j}:=\sum\limits_{\stackrel{t=1}{t\notin \{i,j\}}}^n\omega_t$.
Moreover we obtain
\begin{align*}
\E[\omega_i^2|(\omega_l)_{l\neq i}]&=f_2(S_n^{i}/n) \bigl( 1 + {\mathcal O} (1/n) \bigr)
\end{align*}
with
$$
f_2(x):=f_{2, \beta, K}(x) := \frac{2e^{-\beta}\cosh(2\beta K x)}{1+2e^{-\beta}\cosh(2\beta K x)}
$$
and $S_n^i= S_n - \omega_i$. Also we have that  $0\leq f_i(x)\leq 1$, $i\in\{1,2\}$, for all $x\in\R$.
\end{lemma}

\begin{proof}
First of all we take a look at the conditional probability of $\omega_i$ and $\omega_j$ given all the other spins.
$P_{\beta,K,n}\bigl(\omega_i=t,\omega_j=s|(\omega_l)_{l\neq\{i,j\}}\bigr)$
\begin{eqnarray*}
~&=&\frac{P_{\beta,K,n}\left(\{\omega_i=t,\omega_j=s\}\cap\{(\omega_l)_{l\neq\{i,j\}}\}\right)}{P_{\beta,K,n}\left((\omega_l)_{l\neq\{i,j\}}\right)}\\
&=&\frac{\exp\bigl( -\beta (t^2+s^2)+\frac{\beta K}{n}\bigl( (t+s)^2+2(t+s)S_n^{i,j}\bigr)\bigr)}{\sum\limits_{m,r\in\{-1,0,1\}}\exp\bigl(-\beta (m^2+r^2)+
\frac{\beta K}{n}\bigl((m+r)^2+2(m+r)S_n^{i,j}\bigr)\bigr)}.
\end{eqnarray*}
Let
\begin{align}\label{Dij}
D_{t,s}(x)&:=\exp\bigl(-\beta (t^2+s^2)+\frac{\beta K}{n}\bigl((t+s)^2+2(t+s)x\bigr)\bigr),\\
N_{t,s}(x)&:=t^2s^2D_{t,s}(x).\label{Nij}
\end{align}
Then we have that for the different values of $\omega_i=t$ and $\omega_j=s$
\begin{align*}
\E[\omega_i^2\omega_j^2|(\omega_l)_{l\notin\{i,j\}}]&=\frac{\sum\limits_{t,s\in\{-1,0,1\}}N_{t,s}(S_n^{i,j})}{\sum\limits_{t,s\in\{-1,0,1\}}D_{t,s}(S_n^{i,j})}.
\end{align*} 
We have 9 cases:
For $(\omega_i, \omega_j)=(0,0)$ it is $N_{0,0}(S_n^{i,j})=0$ and $D_{0,0}(S_n^{i,j})=1$. For
$(\omega_i, \omega_j) \in \{ (0,1), (1,0) \}$ we obtain  $N_{0,1}(S_n^{i,j})=0$ and $D_{0,1}(S_n^{i,j})=e^{-\beta+Kn^{-1}}e^{2\beta Kn^{-1}S_n^{i,j}}$.
If $(\omega_i, \omega_j) \in \{ (0,-1), (-1,0) \}$ we obtain $N_{0,-1}(S_n^{i,j})=0$ and $D_{0,-1}(S_n^{i,j})=e^{-\beta+Kn^{-1}}e^{-2\beta Kn^{-1}S_n^{i,j}}$.
Moreover for $(\omega_i, \omega_j)=(1,1)$ it is $ N_{1,1}(S_n^{i,j})=e^{-2\beta+4Kn^{-1}}e^{4\beta Kn^{-1}S_n^{i,j}}$, $D_{1,1}(S_n^{i,j})=N_{1,1}(S_n^{i,j})$, 
for $(\omega_i, \omega_j) \in \{ (1,-1), (-1,1) \}$ we have $N_{1,-1}(S_n^{i,j})=e^{-2\beta}=D_{1,-1}(S_n^{i,j})$. 
Finally for $(\omega_i, \omega_j)=(-1,-1)$ it holds $ N_{-1,-1}(S_n^{i,j})=e^{-2\beta+4Kn^{-1}}e^{-4\beta Kn^{-1}S_n^{i,j}}$ and $D_{-1,-1}(S_n^{i,j})=N_{-1,-1}(S_n^{i,j})$.

Using
$e^{4\beta Kn^{-1} S_n^{i,j}}+e^{-4\beta Kn^{-1} S_n^{i,j}}=2\cosh(4\beta Kn^{-1} S_n^{i,j})$
we have
$\E[\omega_i^2\omega_j^2|(\omega_l)_{l\neq\{i,j\}}]$
\begin{align*}
~&=\frac{e^{-2\beta}+e^{-2\beta}+e^{-2\beta+4\beta Kn^{-1}}\left[e^{4\beta Kn^{-1} S_n^{i,j}}+e^{-4\beta Kn^{-1} S_n^{i,j}}\right]}{1+4e^{-\beta+\beta Kn^{-1}}\cosh(2\beta Kn^{-1} S_n^{i,j})+2e^{-2\beta}+2e^{-2\beta+4\beta Kn^{-1}}\cosh(4\beta Kn^{-1} S_n^{i,j})}\\
&=\frac{2e^{-2\beta}+2e^{-2\beta+4\beta Kn^{-1}}\cosh(4\beta Kn^{-1} S_n^{i,j})}{1+4e^{-\beta+\beta Kn^{-1}}\cosh(2\beta Kn^{-1} S_n^{i,j})+2e^{-2\beta}+2e^{-2\beta+4\beta Kn^{-1}}\cosh(4\beta Kn^{-1} S_n^{i,j})}\\
&=f_1(S_n^{i,j}/n) \bigl( 1 + {\mathcal O} (1/n) \bigr).
\end{align*}
Using the conditional probability of a single spin given all the other spins given in \eqref{condsing} we have
$$
\E\left[\omega_i^2 | (\omega_k)_{k\neq i} \right]=\frac{\exp\bigl(-\beta+\frac{\beta K}{n}\bigl(1-2S_n^i\bigr)\bigr) +\exp\bigl(-\beta +\frac{\beta K}{n}\bigl(1+2S_n^i\bigr)\bigr)}{\sum\limits_{l\in\{-1,0,1\}}\exp\bigl(-\beta l^2+\frac{\beta K}{n}\bigl(l^2+2lS_n^i\bigr)\bigr)}
=f_2(S_n^i/n) \bigl( 1 + {\mathcal O} (1/n) \bigr).
$$
\end{proof}

\begin{lemma}\label{covest}
Let $\gamma=1/2$ if $(\beta,K)\in A$, $\gamma=1/4$ if $(\beta,K)\in B$ and $\gamma=1/6$ if $(\beta,K)=C$. Then we have for $i,j\in\{1,\ldots,n\}$, $i\neq j$,
\begin{align*}
{\rm Cov} (\omega_i^2,\omega_j^2)  = \Oo(1 / n^{\min(4\gamma,1)}).
\end{align*}
where $C$ denotes a constant.
\end{lemma}
\begin{proof}
We have that
\begin{eqnarray*}
\text{Cov}\left(\omega_i^2,\omega_j^2\right)
&=&\E\left[ \E[  \omega_i^2\omega_j^2 |(\omega_l)_{l\neq\{i,j\}}] \right] -\E\left[\E [ \omega_i^2 |(\omega_l)_{l\neq\{i\}} ] \right] \, \E\left[\E [ \omega_j^2 |(\omega_l)_{l\neq\{j\}} ] \right].
\end{eqnarray*}
Using Lemma \ref{doppelspin} we obtain
\begin{eqnarray*}
\text{Cov}\left(\omega_i^2,\omega_j^2\right)
&=& \biggl( \E[f_1(S_n^{i,j}/n)]- \bigl( \E [f_2(S_n^{i}/n)] \bigr)^2 \biggr) \bigl( 1 + {\mathcal O} (1/n) \bigr) 
\end{eqnarray*}
We observe that, for all $i,j\in\{1,\ldots,n\}$, we have
\begin{align*}
f_2(S_n^{i}/n)=f_2(S_n^{i,j}/n)+(f_2(S_n^{i}/n)-f_2(S_n^{i,j}/n))=f_2(S_n^{i,j}/n)+\Oo(n^{-1}).
\end{align*}
This follows since $|f_2(x) - f_2(y)| \leq 2 e^{-\beta} | \cosh( 2 \beta K x) - \cosh (2 \beta K y) | \leq c(\beta,K)|x-y|$, using Lipschitz-continuity of $\cosh(\cdot)$
on a compact interval, where $c(\beta,K)$ is a constant. Hence we obtain
$$
\E[f_1(S_n^{i,j}/n)]- \bigl( \E [f_2(S_n^{i}/n)] \bigr)^2 = \E[f_1(S_n^{i,j}/n)]- \bigl( \E [f_2(S_n^{i,j}/n)] \bigr)^2 + \E [f_2(S_n^{i,j}/n)] {\mathcal O}(1/n) + {\mathcal O}(1/n^2).
$$
We can see that
$$
(f_2(x))^2 =\frac{4e^{-2\beta}\cosh^2(2\beta K x)}{1+4e^{-\beta}\cosh(2\beta K x)+4e^{-2\beta}\cosh^2(2\beta K x)}.
$$
By applying the identity $2\cosh^2(x)=\cosh(2x)+1$ we obtain
$$
(f_2(x))^2 =\frac{2e^{-2\beta}(1+\cosh(4\beta K x))}{1+4e^{-\beta}\cosh(2\beta K x)+2e^{-2\beta}( 1+ \cosh(4\beta K x))} =f_1(x).
$$
Thus
$$
\E[f_1(S_n^{i,j}/n)]- \bigl( \E [f_2(S_n^{i,j}/n)] \bigr)^2 = \E[f_2^2(S_n^{i,j}/n)]- \bigl( \E [f_2(S_n^{i,j}/n)] \bigr)^2 = \V  [f_2(S_n^{i,j}/n)].
$$
Summarising we have
$$
\text{Cov}(\omega_i^2,\omega_j^2) = \biggl( \V  [f_2(S_n^{i,j}/n)] + \E [f_2(S_n^{i,j}/n)] {\mathcal O}(1/n) + {\mathcal O}(1/n^2) \biggr) \bigl( 1 + {\mathcal O} (1/n) \bigr).
$$
Since $f_2'(0)=0$, Taylor expansion of $f_2$ at $0$ leads to
\begin{align*}
f_2(S_n^{i,j}/n) = f_2(0)+\Oo \bigl( (S_n^{i,j}/n)^2 \bigr) = f_2(0)+ \Oo \bigl( W_{\gamma}^2/n^{2\gamma} \bigr) +  \Oo \bigl( W_{\gamma}/n^{\gamma+1} \bigr) + \Oo (1/n^2).
\end{align*}
We note that $\gamma$ depends on the region of $(\beta,K)$: $\gamma=1/2$ if $(\beta,K)\in A$, $\gamma=1/4$ if $(\beta,K)\in B$ and $\gamma=1/6$ if $(\beta,K)=C$.
Finally we obtain, by applying Lemma \ref{MomBEG}, that $\V  [f_2(S_n^{i,j}/n)] = \Oo (n^{-4\gamma})$ and  $\E [f_2(S_n^{i,j}/n)] = \Oo(1)$ and
therefore
\begin{eqnarray*}
\text{Cov}(\omega_i^2,\omega_j^2) = \Oo(1 / n^{\min(4\gamma,1)}).
\end{eqnarray*}
\end{proof}

\begin{remark}\label{nCov}
Note that a proof of Lemma \ref{covest} for parameters $(\beta_n,K_n)$ converging to $(\beta,K)$ taken from one of the regions $A$, $B$ and $C$ follows exactly the lines of the proof of Lemma \ref{covest} with $\beta$ replaced by $\beta_n$ and $K$ replaced by $K_n$. For $(\beta_n,K_n)$ the value of $\gamma$ depends on the region that the sequence is converging from.
\end{remark}

We can bound higher order moments as follows:

\begin{lemma}\label{MomBEG} Let $W_{\gamma}$ be defined in \eqref{Wg}. Then, for all positive bounded sequences $(\beta_n,K_n)$, $\gamma\in(0,1/2]$ and all $l\in\N$ we obtain
\begin{align*}
\E[W_{\gamma}^l]\leq const.(l).
\end{align*}
\end{lemma}
\begin{proof}
The proof is based on the Hubbard-Stratonovich transformation that is used for example in \cite[Lemma 4.1]{CosteniucEllis:2007} to derive the central limit theorem for the total spin per site. The situation of fixed $(\beta,K)$ is included in the study of sequences $(\beta_n, K_n)$ converging to $(\beta,K)$. Let $Y_n$ be a sequence of $N\left(0,(2\beta_nK_n)^{-1}\right)$ random variables independent of all other random variables involved. According to \cite[Lemma 4.1]{CosteniucEllis:2007} we have
\begin{eqnarray*}
\Ll\left(W_{\gamma}+\frac{Y_n}{n^{1/2-\gamma}}\right)=\frac{1}{Z}\exp\left[-nG_{\beta_n,K_n}(y/n^{\gamma})\right],
\end{eqnarray*}
where $Z$ denotes the normalisation. Obviously this transformation does not change the finiteness of any of the moments of $W_{\gamma}$. In order to use a Taylor expansion we have to differentiate between the regions $A$, $B$ and $C$. If the sequence converges to a point in $A$ we have $\gamma=1/2$ and by
 \eqref{taylorGBEGA} the density with respect to the Lebesgue measure is given by const.$\exp(-y^2)$ (up to negligible terms).
 Next we turn to the set $B$. We consider an arbitrary positive, bounded sequence converging to $\beta$ and $K_n$ given by \eqref{KnBEG}.
 With \eqref{taylorGBEGB} we obtain
$$
nG_{\beta_n,K_n}(y/n^{\gamma})=\frac{1}{n^{2\gamma+\Delta_2-1}}\frac{kC_n^{(2)}}{2}y^2+\frac{1}{n^{4\gamma-1}}\frac{G_{\beta_n,K_n}^{(4)}(0)}{24}y^4
+\frac{1}{n^{5\gamma-1}}B_{\beta_n,K_n}(\xi)y^5
$$
where $C_n^{(2)}\rightarrow 2\beta_c$, see \eqref{2ABlGBEG}. In order to obtain a density that is given with respect to the Lebesgue measure 
by const.$\exp(-y^2)$, const.$\exp(-y^4)$ or const.$\exp(-y^2-y^4)$ (up to negligible terms) we note that $v=0$ is required with $v$ defined in Theorem \ref{nB}. It remains to discuss the set $C$. According to \eqref{taylorGBEGC} we have
\begin{eqnarray*}
nG_{\beta_n,K_n}(y/n^{\gamma})&=&\frac{1}{n^{2\gamma+\Delta_2-1}}\frac{kC_n^{(2)}}{2}y^2+\frac{1}{n^{4\gamma+\Delta_1-1}}\frac{bC_n^{(4)}}{24}y^4\\
&&+\frac{1}{n^{6\gamma-1}}\frac{G_{\beta_n,K_n}^{(6)}(0)}{6!}y^6+\frac{1}{n^{7\gamma-1}}C_{\beta_n,K_n}(\xi)y^7,
\end{eqnarray*}
with
$C_n^{(4)}\rightarrow \frac{9}{2}$, see \eqref{4ABlGBEG}. 
In order to obtain a density that is given with respect to the Lebesgue measure by const.$\exp(-G(y))$ (up to negligible terms), where $G(y)$ is a linear combination of the terms $y^2$, $y^4$ and $y^6$, we note that $w=0$ is required with $w$ defined in Theorem \ref{nC}. In each of these cases discussed above a measure with the density stated there has moments of any finite order.
\end{proof}

We now consider the construction of an exchangeable pair $(W_{\gamma},W_{\gamma}')$ in our model
for $W_{\gamma}= \frac{S_n}{n^{1 - \gamma}} = \frac{1}{n^{1 - \gamma}} \sum_{i=1}^n \omega_i$, proving an approximate regression property.
We produce a spin collection $\omega'= (\omega_i')_{i \geq 1}$ via
a {\it Gibbs sampling} procedure: select a  coordinate, say $i$, at random and replace $\omega_i$ by $\omega_i'$ drawn from the
conditional distribution of the $i$'th coordinate given $(\omega_j)_{j \not= i}$, independently from $\omega_i$. Let $I$ be a random variable
taking values $1, 2, \ldots, n$ with equal probability, and independent of all other random variables. Consider
$$
W_{\gamma}' := W_{\gamma} - \frac{\omega_I}{n^{1 - \gamma}} +  \frac{\omega_I'}{n^{1 - \gamma}} = \frac{1}{n^{1 - \gamma}} \sum_{j \not= I} \omega_j + \frac{\omega_I'}{n^{1 - \gamma}}.
$$
Hence $(W_{\gamma},W_{\gamma}')$ is an exchangeable pair and $W_{\gamma}-W_{\gamma}' = \frac{\omega_I - \omega_I'}{n^{1 - \gamma}}$.
For $\mathcal{F} := \sigma(\omega_1, \ldots, \omega_n)$
we obtain
\begin{equation*}
\E[W_{\gamma}-W_{\gamma}'| \mathcal{F}] = \frac{1}{n^{1 - \gamma}} \frac 1n \sum_{i=1}^n \E[ \omega_i - \omega_i' | \mathcal{F}] = \frac{1}{n} \, W_{\gamma} -
\frac{1}{n^{1 - \gamma}} \frac 1n \sum_{i=1}^n  \E[ \omega_i'|\mathcal{F}].
\end{equation*}
With Lemma \ref{BedBEG} and Lemma \ref{GzuBedBEG} we have
\begin{equation*} 
\E[W_{\gamma}-W_{\gamma}'| \mathcal{F}] = \frac{1}{n} \, W_{\gamma} - \frac{1}{n^{1 - \gamma}} \frac 1n \sum_{i=1}^n \bigl( S_n^i/n - \frac{1}{2 \beta K} G_{\beta,K}^{(1)} (S_n^i/n) \bigr) \bigl( 1 + {\mathcal O}(1/n) \bigr).
\end{equation*}
Using 
\begin{equation} \label{sumtrick}
\frac{1}{n^{1 - \gamma}} \frac 1n \sum_{i=1}^n \frac{S_n^i}{n} = \frac 1n \frac{S_n}{n^{1 - \gamma}} - \frac{1}{n^2} \frac{S_n}{n^{1 - \gamma}} = \frac 1n W_{\gamma} - \frac{1}{n^2} W_{\gamma},
\end{equation}
we obtain
\begin{equation} \label{mostim}
\E[W_{\gamma}-W_{\gamma}'| \mathcal{F}] = \frac{1}{n^{1 - \gamma}} \frac 1n \sum_{i=1}^n \frac{1}{2 \beta K} G_{\beta,K}^{(1)} (S_n^i/n)  \bigl( 1 + {\mathcal O}(1/n) \bigr) + {\mathcal O} \bigl( W_{\gamma} / n^2 \bigr).
\end{equation}
Alternatively with $f_{\beta,K}(S_n^i/n) = f_{\beta,K}(S_n/n) + f_{\beta,K}(S_n^i/n) - f_{\beta,K}(S_n/n)$ we see
\begin{equation} \label{mostim2}
\E[W_{\gamma}-W_{\gamma}'| \mathcal{F}] = \frac{1}{n^{1 - \gamma}} \frac{1}{2 \beta K} G_{\beta,K}^{(1)} (S_n/n)  
\bigl( 1 + {\mathcal O}(1/n) \bigr) + {\mathcal O} \bigl( W_{\gamma} / n^2 \bigr) + R_{\beta,K,\gamma}
\end{equation}
with
\begin{equation} \label{rbk}
R_{\beta,K,\gamma} :=  \frac{1}{n^{1 - \gamma}} \frac 1n \sum_{i=1}^n  \bigl( f_{\beta,K}(S_n^i/n) - f_{\beta,K}(S_n/n) \bigr) \bigl( 1 + {\mathcal O}(1/n) \bigr).
\end{equation}

Before proving the theorems we fix an easy but very useful bound on $R_{\beta,K,\gamma}$:

\begin{lemma} \label{rest}
There is a constant $C$ depending only on $\beta$ and $K$ such that
$$
|R_{\beta,K,\gamma} | \leq C \cdot n^{\gamma -2}.
$$
\end{lemma}

\begin{proof}
The denominator of $f_{\beta,K}$ given in \eqref{fBetaBEG} is larger than 1. Hence for any $x,y \in [-1,1]$ we obtain
\begin{eqnarray*}
|f_{\beta,K}(x) - f_{\beta,K}(y)| & \leq & 2 e^{-\beta} | \sinh(2 \beta K x) - \sinh(2 \beta K y) | \\
&& + 4 e^{-2\beta} | \sinh(2 \beta K x) \cosh(2 \beta K y) - \sinh(2 \beta K y)\cosh(2 \beta K x)| \\
& \leq & c(\beta, K) |x-y| +4 e^{-2\beta} | \sinh(2 \beta K (x-y))|,
\end{eqnarray*}
using Lipschitz-continuity of $\sinh(\cdot)$ on a compact interval and the hyperbolic Pythagoras, where $c(\beta,K)$ is a constant.
It follows that
$$
|f_{\beta,K}(S_n^i/n) - f_{\beta,K}(S_n/n)| \leq \frac 1n c(\beta, K) + 4 e^{-2 \beta} \bigl( 2\beta K /n\bigr) + {\mathcal O} (n^{-3}).
$$
\end{proof}

\begin{proof}[Proof of Theorem \ref{Tfixed}]
We will only prove part (1) and (3) of the theorem. The proof of (2) follows the lines of the proof of part (3). 
In each of the cases the exchangeable pair is constructed via a {\it Gibbs sampling}, see \eqref{mostim}. We start with the proof of part (1).
In order to be able to apply Corollary \ref{corsigma} we need the linear regression condition given by \eqref{3.2}. Since $(\beta,K)\in A$ we have $\gamma=1/2$.
With the Taylor expansion of $G_{\beta,K}$ given in \eqref{taylorGBEGA} we have
$$
G_{\beta,K}^{(1)}(S_n^i/n) = \frac{S_n^i}{n}  G_{\beta,K}^{(2)} (0) + {\mathcal O} ((S_n^i/n)^2),
$$ 
and therefore applying \eqref{sumtrick} we obtain with \eqref{mostim}
$$
\E [ W_{1/2} - W_{1/2}' | \F ]  =  \frac 1n \frac{1}{2 \beta K} G_{\beta,K}^{(2)}(0)  W_{1/2} 
 + R_{1/2} =: \frac{\lambda}{\sigma^2} W_{1/2}+R_{1/2}
$$
with
$\lambda= \frac 1n$, $\sigma^2 = \frac{2\beta K}{G_{\beta,K}^{(2)}\left(0\right)}$ and
$$
R_{1/2}={\mathcal O} \bigl( \frac{1}{n \sqrt{n}} W_{1/2}^2 \bigr) + {\mathcal O} \bigl( \frac{1}{n^2} W_{1/2} \bigr).
$$
Hence we have \eqref{3.2} and can apply Corollary \ref{corsigma}. In Remark \ref{koffwahl} we will comment that $\sigma^2$ in the linear regression we found
is not automatically the variance of the limiting normal distribution. By Lemma \ref{MomBEG} we know that $\E (W_{1/2}^2)$ is bounded and therefore
$\lambda^{-1} \sqrt{\E(R_{1/2}^2)} = {\mathcal O}(n^{-1/2})$. Therefore the last three terms in \eqref{mainbound3} can be bounded by a constant (depending on $\beta$
and $K$) times $n^{-1/2}$. Next we have to consider the variance of
\begin{eqnarray}\label{Ai}\nonumber
\E\left[(W_{1/2}'-W_{1/2})^2|W_{1/2}\right]&=&\frac{1}{n^2}\sum\limits_{i=1}^n\E\left[(\omega_i'-\omega_i)^2|W_{1/2}\right]\\\nonumber
&=&\frac{1}{n^2}\sum\limits_{i=1}^n\bigl(\omega_i^2+\E\left[\omega_i'^2|W_{1/2}\right]-2\omega_i\E\left[\omega_i'|W_{1/2}\right]\bigr)\\
&=:&A_1+A_2+A_3.
\end{eqnarray}
To get an estimate of the variance of this expression we will bound the variances of the $A_i$ and start by taking a look at the variance of $A_1$.
\begin{eqnarray*}
\V\left[A_1\right]&=&\frac{1}{n^4}\sum\limits_{i=1}^n\V\left[\omega_i^2\right]+\frac{1}{n^4}\sum\limits_{1\leq i<j\leq n}\text{Cov}\left(\omega_i^2,\omega_j^2\right).
\end{eqnarray*}
Applying Lemma \ref{covest} with $\gamma=1/2$ we have $\cov(\omega_i^2,\omega_j^2)=\Oo(n^{-1})$.
This leads to the bound $\V\left[A_1\right]=\Oo\left(n^{-3}\right)$.
A conditional version of Jensen's inequality yields $\V\left[A_2\right]\leq\V\left[A_1\right]$. Thus the variance of $A_2$ has the same order as the variance of $A_1$. Furthermore Lemma \ref{BedBEG}, Lemma \ref{GzuBedBEG} and Lemma
\ref{MomBEG} yield
$$
\frac{1}{2} |A_3| =\biggl|\frac{1}{n^2}\sum\limits_{i=1}^n\omega_i\left[f_{\beta,K}(S_n^i/n) \bigl( 1 + {\mathcal O}(1/n) \bigr) \right]\biggr|
= \Oo \bigl( \frac{W_{1/2}^2}{n^2} \bigr) + \Oo\bigl( \frac{1}{n^2} \bigr).
$$
As a result of Lemma \ref{MomBEG} the variance of $A_3$ can be bounded by a constant times $n^{-4}$. Summarising these estimations the variance of $\E\bigl[(W_{1/2}'-W_{1/2})^2|W_{1/2}\bigr]$ can be bounded by $9$ times the maximum of the variances of the terms $A_1$, $A_2$ and $A_3$, which is a constant times $n^{-3}$. Thus, finally
\begin{eqnarray*}
\frac{\sigma^2}{2\lambda}\sqrt{\V\left[\E\bigl[(W_{1/2}'-W_{1/2})^2|W_{1/2}\bigr]\right]}=\Oo\left(n^{-1/2}\right),
\end{eqnarray*}
which completes the proof for the region $A$.

Next we turn to the region $C$, part (3) of Theorem \ref{Tfixed}. As has been said before, the proof for region $B$ follows the steps of the next lines except that slight changes regarding the Taylor expansion have to be made. In order to apply Theorem \ref{generaldensity2} we have to show that the linear regression condition \eqref{exchangepsi} is fulfilled. Applying the Taylor expansion of $G_{\beta_c, K_c(\beta_c)}$ in \eqref{taylorGBEGC}, with \eqref{mostim2} and \eqref{rbk} we obtain 
\begin{eqnarray*}
\E [ W_{1/6} - W_{1/6}' | \F ] & = & \frac{1}{n^{5/3}} \frac{G_{\beta_c,K_c(\beta_c)}^{(6)}(0)}{5! \, 2\beta_c K_c(\beta_c)} W_{1/6}^5 + 
{\mathcal O} \bigl( \frac{W_{1/6}^5}{n^{8/3}} + \frac{W_{1/6}^6}{n^{11/6}} \bigr) + R_{\beta_c,K_c(\beta_c), 1/6}.
\end{eqnarray*}
Thus
\begin{eqnarray*}
\E [ W_{1/6} - W_{1/6}' | \F ] & = & - \lambda \psi(W_{1/6}) + R_{1/6}
\end{eqnarray*}
with
$\lambda=n^{-5/3}$ and $\psi(x) = -\frac{G_{\beta_c,K_c(\beta_c)}^{(6)}\left(0\right)}{5! \, 2 \beta_c K_c(\beta_c)} x^5$
and $R_{1/6}= R_{\beta_c,K_c(\beta_c), 1/6} + {\mathcal O} \bigl( n^{-11/6} \bigr)$, where we have used Lemma \ref{MomBEG}.

Note that $\frac{\psi(x)}{\E \bigl( W_{1/6} \psi(W_{1/6}) \bigr)} =\frac{x^5}{\E(W_{1/6}^6)}$. Applying Theorem \ref{generaldensity2} we will compare the distribution of $W_{1/6}$ with a distribution with Lebesgue-probability density proportional to $\exp \bigl( - \frac{x^6}{6 \E(W^6)} \bigr)$. This density as well
as the density $p$ with $\psi = p'/p$ fulfill assumption (B), see \cite[Lemma 2.2]{Eichelsbacher/Loewe:2010} 
Given the linear regression condition we are able to compute the bound given in Theorem \ref{generaldensity2}. Since $|W_{1/6}-W_{1/6}'|\leq\frac{1}{n^{5/6}}$
we set $A :=\frac{1}{n^{5/6}}$ and obtain with the help of Lemma \ref{MomBEG}
\begin{eqnarray*}
\frac{1}{\lambda}\left(\frac{d_4A^3}{4}\right)+\frac{3A}{2}\E\abs{\psi(W_{1/6})}=\Oo\left(n^{-5/6}\right).
\end{eqnarray*}
The last term in \eqref{kolall2} is zero. From Lemma \ref{rest} we know that $|R_{\beta_c,K_c(\beta_c), 1/6}| = {\mathcal O}(n^{-11/6})$
and using Lemma \ref{MomBEG}, we obtain that 
\begin{eqnarray*}
\left(d_1+ d_2 \sqrt{ \E(W^2)} + \frac{3}{2}A\right)\lambda^{-1} \sqrt{\E[R_{1/6}^2]}&=&\Oo\left(n^{-1/6}\right).
\end{eqnarray*}
Additionally
\begin{eqnarray*}
\E\left[(W_{1/6}-W_{1/6}')^2|W_{1/6}\right]&=&\frac{1}{n^{2/3}}\left(A_1+A_2+A_3\right)
\end{eqnarray*}
with the $A_i$'s defined in \eqref{Ai}.
In order to be able to estimate the variance of this expression again we will bound the variances of the $A_i$'s.
With Lemma \ref{covest} for $\gamma=1/6$ (since $(\beta,K) \in C$) we have $\text{Cov}(\omega_i^2,\omega_j^2)=\Oo(n^{-2/3})$
and therefore $\V\left[A_1\right]=\Oo\left(n^{-8/3}\right)$. A conditional version of Jensen's inequality yields $\V\left[A_2\right]\leq\V\left[A_1\right]$. 
Thus $\V\left[n^{-2/3}A_1\right]=\V\left[n^{-2/3}A_2\right]=\Oo(n^{-4})$. Furthermore Lemma \ref{BedBEG}, Lemma \ref{GzuBedBEG} and Lemma \ref{MomBEG} yield
\begin{eqnarray*}
\frac{1}{2}\frac{1}{n^{2/3}} |A_3| =\biggl|\frac{1}{n^{8/3}}\sum\limits_{i=1}^n\omega_i\left[f_{\beta,K}(S_n^i/n) \bigl( 1 + {\mathcal O}(1/n) \bigr) \right]\biggr|
= \Oo \bigl( \frac{W_{1/6}^2}{n^{8/3}} \bigr) + \Oo\bigl( \frac{1}{n^2} \bigr).
\end{eqnarray*}
As a result of Lemma \ref{MomBEG} the variance of $A_3$ can be bounded by a constant times $n^{-4}$. Summarising the variance of $\E\bigl[(W_{1/6}-W_{1/6}')^2|W_{1/6}\bigr]$ can be bounded by $9$ times the maximum of the variances of the terms $\frac{1}{n^{2/3}}A_1$, $\frac{1}{n^{2/3}}A_2$ and $\frac{1}{n^{2/3}}A_3$, which is a constant times $n^{-4}$. Thus, finally
\begin{eqnarray*}
\frac{d_2}{2 \lambda}\sqrt{\V\left[\E\left[(W_{1/6}-W_{1/6}')^2|W_{1/6}\right]\right]}=\Oo\left(n^{-1/3}\right),
\end{eqnarray*}
which completes the proof for the region $C$.
\end{proof}

\begin{remark} \label{koffwahl}
Note that $\sigma^2$ in the proof of part (1) is not the limiting variance in \cite[Theorem 5.5]{CosteniucEllis:2007}. The variance is $(G_{\beta,K}^{(2)}(0))^{-1} - (2 \beta K)^{-1}$ with $G_{\beta,K}^{(2)}(0) =2\beta K ( 1 - 2 \beta K c_{\beta}^{(2)}(0))$. Interesting enough in the classical Curie-Weiss model, the
prefactor $\sigma^2$ in the regression identity coincides with the limiting variance, see the proof of Theorem 3.7 in \cite{Eichelsbacher/Loewe:2010}:
Here the limiting variance is $(1 - \beta)^{-1}$ and the prefactor is $\beta / G_{\beta}^{(2)}(0) = \beta/ (\beta (1 - \beta c_{\beta}^{(2)}(0))$. In the Curie-Weiss
model one has $c_{\beta}^{(2)}(0)=1$ and hence $\beta / G_{\beta}^{(2)}(0)= (1- \beta)^{-1}$. 
\end{remark}

\begin{proof}[Proof of Theorem \ref{nA}]
Since zero is a unique minimum for the whole set $A$ the proof requires exactly the same steps as the proof of part a) of Theorem \ref{Tfixed}.
\end{proof}

Now we turn to the theorems involving the sequence $(\beta_n,K_n)$ that converges to $(\beta,K)$.

\begin{proof}[Proof of Theorem \ref{nB}]
Our goal is to apply either Corollary \ref{corsigma} or Theorem \ref{generaldensity2}, depending on whether there is a Gaussian or a non-Gaussian limit. Given $W_{\gamma}$ again we construct a coupling $W_{\gamma}'$ via {\it Gibbs sampling} such that $(W_{\gamma}, W_{\gamma}')$ is exchangeable. This will be used in order to calculate $\lambda$ and $R$ to get the linear regression condition, which is, also due to the form of the limit density, either taken from \eqref{exchangepsi} or \eqref{3.2}. Let $\F:=\sigma(\omega_1,\ldots,\omega_n)$. We start with $\eqref{mostim2}$ and plug in the Taylor expansion $\eqref{taylorGBEGB}$ to obtain
\begin{eqnarray*}
\E [ W_{\gamma} - W_{\gamma}' | \F ] & = & \frac{1}{2\beta_n K_n}G_{\beta_n,K_n}^{(2)}\left(0\right) \frac 1n W_{\gamma}
+\frac{1}{3!2\beta_n K_n}G_{\beta_n,K_n}^{(4)}\left(0\right) \frac{1}{n^{1+2\gamma}} W_{\gamma}^3 + R_{\gamma}
\end{eqnarray*}
with 
\begin{equation} \label{rgamma}
R_{\gamma} := \Oo \bigl( \frac{W_{\gamma}}{n^2} \bigr) + \Oo \bigl( \frac{W_{\gamma}^4}{n^{1 + 3 \gamma}} \bigr) + R_{\beta_n, K_n, \gamma},
\end{equation}
where $R_{\beta_n, K_n, \gamma}$ is defined in \eqref{rbk}.
We can use \eqref{2ABlGBEG} to obtain
\begin{eqnarray} \label{dis1}
\E [ W_{\gamma} - W_{\gamma}' | \F ] & = & \frac{k}{K_c(\beta_n)} \frac{1}{n^{1+\Delta_2}} W_{\gamma}+\frac{1}{3!2\beta_n K_n} G_{\beta_n,K_n}^{(4)}\left(0\right) \frac{1}{n^{1+2\gamma}}W_{\gamma}^3 + R_{\gamma}.
\end{eqnarray}

Proof of part (1): 
Depending on the influence of regions $A$ and $B$ there are different expressions for $\lambda$ due to an application of either Theorem \ref{generaldensity2} or Corollary \ref{corsigma}. 
We note that   the choice $\Delta_2=2\gamma$ seem to be necessary to get  the expressions prior to $W_{\gamma}$ and $W_{\gamma}^3$ of the same order. 
We do not expect the choice $\gamma = 1/2$ or $\gamma = 1/6$: remember that the first summand on the right hand side of \eqref{dis1} lead to a Gaussian
limit in the case of choosing the scaling $1/n$ whereas the second summand lead to a limiting density in the case of the scaling $1/n^{5/3}$. Hence $1/n^2$ should be expected to be overdesigned whereas $5/3=1 + 2 \gamma$ gives $\gamma=1/3$, which is at least a non-classical scaling. Hence we consider $\gamma=1/4$ and $\Delta_2=1/2$
and expect an influence of both regions $A$ and $B$ (and it is known from \cite[Theorem 6.1]{CosteniucEllis:2007} that this is the right choice).
So we end up with
\begin{eqnarray*}
\E [ W_{1/4} - W_{1/4}' | \F ] & = & - \lambda \psi(W_{1/4}) + R_{1/4}
\end{eqnarray*}
with
$\lambda=\frac{1}{n^{3/2}}$ and $\psi(x)=- \frac{k}{K_c(\beta_n)} x - \frac{1}{3!2\beta_n K_n}  G_{\beta_n,K_n}^{(4)}\left(0\right)  x^3$
and $R_{1/4} = R_{\beta_n, K_n, 1/4} + \Oo (n^{-7/4})$, where we used Lemma \ref{MomBEG}.

Note that $\frac{\psi(x)}{\E \bigl( W_{1/4} \psi(W_{1/4}) \bigr)} =  \frac{c_1 x + c_2  x^3}{c_3}$ with explicit formulas for 
$c_1 = c_1(\beta_n, K_n)$, $c_2 = c_2(\beta_n, K_n)$ and $c_3= c_3(\beta_n, K_n, \E(W_{1/4}^2), \E(W_{1/4}^4))$. Applying Theorem \ref{generaldensity2} we will compare the distribution of $W_{1/4}$ with a distribution with Lebesgue-probability density proportional to $\exp \bigl( - \frac{c_1 x^2/2}{c_3}  -\frac{ c_2 x^4/4}{c_3} \bigr)$. This density as well
as the density $p$ with $\psi = p'/p$ fulfil assumption (B), see \cite[Lemma 2.2]{Eichelsbacher/Loewe:2010} 
Given the linear regression condition we are able to compute the bound given in Theorem \ref{generaldensity2}.
Since $|W_{1/4}-W_{1/4}'|\leq\frac{1}{n^{3/4}}$
we set $A :=\frac{1}{n^{3/4}}$ and obtain with the help of Lemma \ref{MomBEG}
\begin{eqnarray*}
\frac{1}{\lambda}\left(\frac{d_4A^3}{4}\right)+\frac{3A}{2}\E\abs{\psi(W_{1/4})}=\Oo\left(n^{-3/4}\right).
\end{eqnarray*}
The last term in \eqref{kolall2} is zero. From Lemma \ref{rest} we know that $|R_{\beta_n,K_n, 1/4}| = {\mathcal O}(n^{-7/4})$
and using Lemma \ref{MomBEG}, we obtain that 
\begin{eqnarray*}
\left(d_1+ d_2 \sqrt{ \E(W^2)} + \frac{3}{2}A\right)\lambda^{-1} \sqrt{\E[R_{1/4}^2]}&=&\Oo\left(n^{-1/4}\right).
\end{eqnarray*}
Additionally
\begin{eqnarray*}
\E\left[(W_{1/4}-W_{1/4}')^2|W_{1/4}\right]&=&\frac{1}{n^{1/2}}\left(A_1+A_2+A_3\right)
\end{eqnarray*}
with the $A_i$'s defined in \eqref{Ai} except that the expectation is now taken for the measure $P_{n,\beta_n,K_n}$.
In order to be able to estimate the variance of this expression again we will bound the variances of the $A_i$'s.
Keeping Remark \ref{nCov} in mind we can apply Lemma \ref{covest} for $\gamma=1/4$ and get $\text{Cov}(\omega_i^2,\omega_j^2)=\Oo(n^{-1})$
and therefore $\V\left[A_1\right]=\Oo\left(n^{-3}\right)$. A conditional version of Jensen's inequality yields $\V\left[A_2\right]\leq\V\left[A_1\right]$. 
Thus $\V\left[n^{-1/2}A_1\right]=\V\left[n^{-1/2}A_2\right]=\Oo(n^{-4})$. Furthermore Lemma \ref{BedBEG} and Lemma \ref{GzuBedBEG} yield
\begin{eqnarray*}
\frac{1}{2}\frac{1}{n^{1/2}} |A_3| =\biggl|\frac{1}{n^{5/2}}\sum\limits_{i=1}^n\omega_i\left[f_{\beta,K}(S_n^i/n) \bigl( 1 + {\mathcal O}(1/n) \bigr) \right]\biggr|
= \Oo \bigl( \frac{W_{1/4}^2}{n^{5/2}} \bigr) + \Oo\bigl( \frac{1}{n^2} \bigr).
\end{eqnarray*}
As a result of Lemma \ref{MomBEG} the variance of $A_3$ can be bounded by a constant times $n^{-4}$. Summarising the variance of $\E\bigl[(W_{1/4}-W_{1/4}')^2|W_{1/4}\bigr]$ can be bounded by $9$ times the maximum of the variances of the terms $\frac{1}{n^{1/2}}A_1$, $\frac{1}{n^{1/2}}A_2$ and $\frac{1}{n^{1/2}}A_3$, which is a constant times $n^{-4}$. Thus, finally
\begin{eqnarray*}
\frac{d_2}{2 \lambda}\sqrt{\V\left[\E\left[(W_{1/4}-W_{1/4}')^2|W_{1/4}\right]\right]}=\Oo\left(n^{-1/2}\right),
\end{eqnarray*}
which completes the proof of part (1) where both regions $A$ and $B$ influence the limit distribution.

Proof of part (2):
If the influence is from region $A$ the term involving $W_{\gamma}^3$ has to be of smaller order than the term of $W_{\gamma}$. Hence, we note that the condition $\Delta_2<2\gamma$ has to be fulfilled. In \cite[Theorem 6.1]{CosteniucEllis:2007} it is proved that the only interesting choice for $\gamma$ and
 $\Delta_2$ is to take $\Delta_2 \in (0, 1/2)$, $\gamma \in (1/4, 1/2)$ and $1 -2 \gamma = \Delta_2$. Let us discuss this case in detail.
We consider the linear regression condition 
$$
\E [ W_{\gamma}- W_{\gamma}' | \F ] = \frac{k}{K_c(\beta_n)} \frac{1}{n^{1 + \Delta_2}} W_{\gamma}+ \tilde{R}_{\gamma} 
$$
with
$$
\tilde{R}_{\gamma} = \frac{1}{3!2\beta_n K_n}G_{\beta_n,K_n}^{(4)}\left(0\right) \frac{1}{n^{1+2\gamma}} W_{\gamma}^3 + R_{\gamma}
$$
and $R_{\gamma}$ given in \eqref{rgamma}. Hence $\lambda=\frac{1}{n^{1+\Delta_2}}$ and $\psi(x) = - \frac{k}{K_c(\beta_n)} x =: -x/ \sigma^2$ and
we compare the distribution of $W_{\gamma}$ with a $N(0, \E(W_{\gamma}^2))$ distribution.
Since $|W_{\gamma}-W_{\gamma}'|\leq\frac{1}{n^{1 - \gamma}}$
we set $A :=\frac{1}{n^{1 -\gamma}}$ and obtain with the help of Lemma \ref{MomBEG} that $\sigma^2 1.5 A \sqrt{\E(W_{\gamma})} = \Oo (n^{\gamma-1})$.
With $A^3/ \lambda = n^{\Delta_2 -2 + 3 \gamma} = n^{\gamma-1}$ the second last term in \eqref{mainbound3} has the same order.
From Lemma \ref{rest} we know that $|R_{\beta_n,K_n, \gamma}| = {\mathcal O}(n^{\gamma-2})$ and therefore $\Oo(\tilde{R}_{\gamma}) =
n^{- \min(2-\gamma, 1+2 \gamma)}$,
using Lemma \ref{MomBEG}.  Summarising we have 
\begin{eqnarray*}
\lambda^{-1} \sqrt{\E[\tilde{R}_{\gamma}^2]}&=&\Oo\left(n^{-\min(\gamma, 4\gamma -1)}\right).
\end{eqnarray*}
As we can see, the third order term of the Taylor expansion of $G_{\beta_n,K_n}$ now influences the order of the remainder.
We have
\begin{eqnarray*}
\E\left[(W_{\gamma}-W_{\gamma}')^2|W_{\gamma}\right]&=&\frac{1}{n^{1 - 2 \gamma}}\left(A_1+A_2+A_3\right)
\end{eqnarray*}
with the $A_i$'s defined in \eqref{Ai}. We can apply Lemma \ref{covest} to obtain $\text{Cov}(\omega_i^2,\omega_j^2)=\Oo(n^{-\min(4 \gamma, 1)})$
and therefore as seen before $\V\left[n^{-(1 - 2\gamma)}A_1\right]=\V\left[n^{-(1- 2\gamma)}A_2\right]=\Oo(n^{- \min(4, 5 - 4 \gamma)}) = \Oo(n^{- 5 + 4 \gamma})$. 
Furthermore Lemma \ref{BedBEG} and Lemma \ref{GzuBedBEG} yield
$
\frac{1}{2}\frac{1}{n^{1-2 \gamma}} |A_3| = \Oo \bigl( \frac{W_{\gamma}^2}{n^{3-2 \gamma}} \bigr) + \Oo\bigl( \frac{1}{n^2} \bigr)$.
As a result of Lemma \ref{MomBEG} the variance of $A_3$ can be bounded by a constant times $n^{-4}$. Summarising the variance of $\E\bigl[(W_{\gamma}-W_{\gamma}')^2|W_{\gamma}\bigr]$ can be bounded by a constant times $n^{-5+4 \gamma}$. Thus, finally
\begin{eqnarray*}
\frac{\sigma^2}{2 \lambda}\sqrt{\V\left[\E\left[(W_{\gamma}-W_{\gamma}')^2|W_{\gamma}\right]\right]}=\Oo\left(n^{2 - 2\gamma -5/2 + 2\gamma}\right) = \Oo\left(n^{-1/2}\right).
\end{eqnarray*}
The case $\Delta_2 \in(0,1/2)$ corresponds to the slowest convergence of $K_n \to K_c(\beta)$ with $(\beta, K_c(\beta)) \in B$, in which only $A$
influences the form of the limiting distribution, which has a Gaussian density even though a non-classical scaling is given by $n^{1-\gamma}$. 
We obtain an additional and remarkable phenomenon: for any $\gamma \in (1/4, 1/3]$ the rate of convergence is $1 / n^{4 \gamma-1}$ 
whereas for all $\gamma \in [1/3, 1/2)$ we obtain the rate $1/n^{\gamma}$.

Proof of part (3):
Finally we consider the case which corresponds to the largest value of $\Delta_2$, namely $\Delta_2>2\gamma$. We take $\Delta_2 > 1/2$ and $\gamma =1/4$ and thus the most rapid convergence of $K_n \to K_c(\beta)$. 
Now we end up with
\begin{eqnarray*}
\E [ W_{1/4} - W_{1/4}' | \F ] & = & - \lambda \psi(W_{1/4}) + \tilde{R}_{1/4}
\end{eqnarray*}
with
$\lambda=\frac{1}{n^{3/2}}$ and $\psi(x)= - \frac{1}{3!2\beta_n K_n}  G_{\beta_n,K_n}^{(4)}\left(0\right)  x^3$
and $ \tilde{R}_{1/4} = R_{\beta_n, K_n, 1/4} + \frac{k}{K_c(\beta_n)} \frac{1}{n^{1 + \Delta_2}} W_{1/4} + \Oo (n^{-7/4})$, where we used Lemma \ref{MomBEG}.
Again $|W_{1/4}-W_{1/4}'|\leq\frac{1}{n^{3/4}}$, we set $A :=\frac{1}{n^{3/4}}$ and obtain with the help of Lemma \ref{MomBEG}, that
the first summand in \eqref{kolall2} is of order $\Oo\left(n^{-1/2}\right)$ and third term of order $\Oo\left(n^{-3/4}\right)$.
From Lemma \ref{rest} we know that $|R_{\beta_n,K_n, 1/4}| = {\mathcal O}(n^{-7/4})$. The second summand of $\tilde{R}_{1/4}$ is of order 
${\mathcal O}(n^{-(1 + \Delta_2)})$. Using Lemma \ref{MomBEG}, we obtain that 
$\left(d_1+ d_2 \sqrt{ \E(W^2)} + \frac{3}{2}A\right)\lambda^{-1} \sqrt{\E[\tilde{R}_{1/4}^2]}=\Oo\left(n^{-\min(1/4, \Delta_2-1/2)} \right)$.
This is also an interesting phase transition: for any $\Delta_2 \in (1/2, 3/4)$ we obtain a slow rate of convergence $n^{-(\Delta_2 -1/2)}$, but
when $K_n$ converges more rapid in the sense of $\Delta_2 \geq 3/4$, we obtain the rate $n^{-1/4}$. The proof via Stein's method gives the 
information that in case (3), we have to assume $\Delta_2 > 1/2$.
\end{proof}

\begin{proof}[Proof of Theorem \ref{nC}]
Again our goal is to apply either Corollary \ref{corsigma} or Theorem \ref{generaldensity2}, depending on whether there is a Gaussian or a non-Gaussian limit. Given $W_{\gamma}$ again we construct a coupling $W_{\gamma}'$ via {\it Gibbs sampling} such that $(W_{\gamma}, W_{\gamma}')$ is exchangeable. 
We start with $\eqref{mostim2}$ and plug in the Taylor expansion $\eqref{taylorGBEGC}$ to obtain
\begin{eqnarray*} \label{tayc}
\E [ W_{\gamma} - W_{\gamma}' | \F ] & = & \frac{1}{2\beta_n K_n}G_{\beta_n,K_n}^{(2)}\left(0\right) \frac 1n W_{\gamma}
+\frac{1}{3!2\beta_n K_n}G_{\beta_n,K_n}^{(4)}\left(0\right) \frac{1}{n^{1+2\gamma}} W_{\gamma}^3 \\ & + &\frac{1}{5! 2 \beta_n K_n}G_{\beta_n,K_n}^{(6)}\left(0\right) \frac{1}{n^{1+4\gamma}} W_{\gamma}^5 + R_{\gamma}
\end{eqnarray*}
with 
\begin{equation} \label{rgammaC}
R_{\gamma} := \Oo \bigl( \frac{W_{\gamma}}{n^2} \bigr) + \Oo \bigl( \frac{W_{\gamma}^6}{n^{1 + 5 \gamma}} \bigr) + R_{\beta_n, K_n, \gamma},
\end{equation}
where $R_{\beta_n, K_n, \gamma}$ is defined in \eqref{rbk}.
We can use \eqref{2ABlGBEG} and \eqref{4ABlGBEG} to obtain
\begin{eqnarray*} \label{dis1}
\E [ W_{\gamma} - W_{\gamma}' | \F ] & = & \frac{k}{K_c(\beta_n)} \frac{1}{n^{1+\Delta_2}} W_{\gamma}+\frac{b C_n^{(4)}}{3!2\beta_n K_n} 
\frac{1}{n^{1+2\gamma+ \Delta_1}}W_{\gamma}^3 \\ &+& \frac{1}{5! 2 \beta_n K_n}G_{\beta_n,K_n}^{(6)}\left(0\right) \frac{1}{n^{1+4\gamma}} W_{\gamma}^5 + R_{\gamma} \\
& =:& T_1 + T_2 + T_3 +  R_{\gamma}.
\end{eqnarray*}

Proof of part (1): we consider $\gamma=1/6$, $\Delta_1=1/3$ and $\Delta_2=2/3$ and get
\begin{eqnarray*}
\E [ W_{1/6} - W_{1/6}' | \F ] & = & - \lambda \psi(W_{1/6}) + R_{1/6}
\end{eqnarray*}
with
$\lambda=\frac{1}{n^{5/3}}$ and $\psi(x)=- \frac{k}{K_c(\beta_n)} x - \frac{b C_n^{(4)}}{3!2\beta_n K_n}  x^3 - \frac{G_{\beta_n, K_n}^{(6)}(0)}{5! 2 \beta_n K_n} x^5$ and $R_{1/6} = R_{\beta_n, K_n, 1/6} + \Oo (n^{-11/6})$, where we used Lemma \ref{MomBEG}.
Note that $\frac{\psi(x)}{\E ( W_{1/6} \psi(W_{1/6}) )} =  \frac{c_1 x + c_2  x^3 + c_3 x^5}{c_4}$ with explicit formulas for 
$c_1 = c_1(\beta_n, K_n)$, $c_2 = c_2(\beta_n, K_n)$, $c_3= c_3(\beta_n, K_n)$ and $c_4= c_4(\beta_n, K_n, \E(W_{1/6}^2), \E(W_{1/6}^4), \E(W_{1/6}^6))$. Applying Theorem \ref{generaldensity2} we will compare the distribution of $W_{1/6}$ with a distribution with Lebesgue-probability density proportional to $\exp \bigl( - \frac{c_1 x^2/2}{c_4}  -\frac{ c_2 x^4/4}{c_4} - \frac{c_3 x^6/6}{c_4} \bigr)$. This density as well
as the density $p$ with $\psi = p'/p$ fulfil assumption (B), see \cite[Lemma 2.2]{Eichelsbacher/Loewe:2010} 
Given the linear regression condition we are able to compute the bound given in Theorem \ref{generaldensity2}.
Since $|W_{1/6}-W_{1/6}'|\leq\frac{1}{n^{5/6}}$ we obtain with $A = n^{-5/6}$ and the help of Lemma \ref{MomBEG}
$\frac{1}{\lambda}\left(\frac{d_4A^3}{4}\right)+\frac{3A}{2}\E\abs{\psi(W_{1/6})}=\Oo\left(n^{-5/6}\right)$
and $\left(d_1+ d_2 \sqrt{ \E(W^2)} + \frac{3}{2}A\right)\lambda^{-1} \sqrt{\E[R_{1/6}^2]}=\Oo\left(n^{-1/6}\right)$. Exactly as in the
proof of part (3) of Theorem \ref{Tfixed} we have
\begin{eqnarray*}
\frac{d_2}{2 \lambda}\sqrt{\V\left[\E\left[(W_{1/6}-W_{1/6}')^2|W_{1/6}\right]\right]}=\Oo\left(n^{-1/3}\right),
\end{eqnarray*}
which completes the proof of part (1) where all regions $A$, $B$ and $C$ influence the limit distribution.

Proof of part (2):
We consider the linear regression condition 
$$
\E [ W_{\gamma}- W_{\gamma}' | \F ] = \frac{k}{K_c(\beta_n)} \frac{1}{n^{1 + \Delta_2}} W_{\gamma}+ \tilde{R}_{\gamma} 
$$
with $\tilde{R}_{\gamma} = T_2 + T_3 +  R_{\gamma}$, where $R_{\gamma}$ is defined in \eqref{rgammaC}.
Hence $\lambda=\frac{1}{n^{1+\Delta_2}}$ and $\psi(x) = - \frac{k}{K_c(\beta_n)} x =: -x/ \sigma^2$ and
we compare the distribution of $W_{\gamma}$ with a $N(0, \E(W_{\gamma}^2))$ distribution.
With $|W_{\gamma}-W_{\gamma}'|\leq\frac{1}{n^{1 - \gamma}}$ and $A :=\frac{1}{n^{1 -\gamma}}$ we have $\sigma^2 1.5 A \sqrt{\E(W_{\gamma})} = \Oo (n^{\gamma-1})$ and $A^3/ \lambda = n^{\gamma-1}$. From Lemma \ref{rest} we know that $|R_{\beta_n,K_n, \gamma}| = {\mathcal O}(n^{\gamma-2})$ and therefore $\Oo(\tilde{R}_{\gamma}) = n^{- \min(2-\gamma, 1+4 \gamma, 1 + 2 \gamma + \Delta_1)}$, using Lemma \ref{MomBEG}.  
With the proof of part (2) of Theorem \ref{nB} we have
\begin{eqnarray*}
\frac{\sigma^2}{2 \lambda}\sqrt{\V\left[\E\left[(W_{\gamma}-W_{\gamma}')^2|W_{\gamma}\right]\right]}= \Oo\left(n^{-1/2}\right)
\end{eqnarray*}
and therefore the leading order is given by the order of $\lambda^{-1} \sqrt{\E[\tilde{R}_{\gamma}^2]}$, which leads to the three different cases, solving
the minimization problem $\min(2-\gamma, 1+4 \gamma, 1 + 2 \gamma + \Delta_1)$ for $\gamma \in (1/4, 1/2)$ and $\Delta_1 >0$.

Proof of part (3): We have exactly the same situation as in part (2). Therefore one has to solve the minimization problem $\min(2-\gamma, 1+4 \gamma, 1 + 2 \gamma + \Delta_1)$ with $\gamma \in (1/6, 1/4]$ and $\Delta_1 > 2 \Delta_2-1 = 1 - 4 \gamma$. This leads to the four cases stated in the Theorem.

Proof of part (4): With $\gamma=1/6$, $\Delta_1>1/3$ and $\Delta_2 >2/3$ 
We consider the linear regression condition 
$$
\E [ W_{1/6}- W_{1/6}' | \F ] = \frac{1}{5! 2 \beta_n K_n}G_{\beta_n,K_n}^{(6)}\left(0\right) \frac{1}{n^{5/3}} W_{1/6}^5 + \tilde{R}_{1/6} 
$$
with $\tilde{R}_{1/6} = T_1 + T_2 +  R_{1/6}$, where $R_{1/6}$ is defined in \eqref{rgammaC}.
Hence $\lambda=\frac{1}{n^{5/3}}$ and $\psi(x) = -  \frac{1}{5! 2 \beta_n K_n}G_{\beta_n,K_n}^{(6)}\left(0\right) x^5$ and
we compare the distribution of $W_{1/6}$ with a distribution with Lebesgue-density proportional to $\exp \bigl( - \frac{x^6}{6 \E(W^6)} \bigr)$.
As in the proofs of part (2) and (3) we see that the leading order is the order of $\lambda^{-1} \sqrt{\E[\tilde{R}_{1/6}^2]}$. We now have to solve
the minimization problem $\min(11/6,  \Delta_1+ 4/3, \Delta_2 +1)$ for $\Delta_1, \Delta_2 >0$ which leads to the result stated in the Theorem.

Proof of part (5):
With $4\gamma = 1 - \Delta_1$, $\gamma  \in (1/6, 1/4)$ and $2 \Delta_2 > \Delta_1 +1$ we consider the linear regression condition 
$$
\E [ W_{\gamma}- W_{\gamma}' | \F ] =\frac{1}{3!2\beta_n K_n}G_{\beta_n,K_n}^{(4)}\left(0\right) \frac{1}{n^{1+2\gamma}} W_{\gamma}^3
+ \tilde{R}_{\gamma} 
$$
with $\tilde{R}_{\gamma} = T_1+ T_3 +  R_{\gamma}$, where $R_{\gamma}$ is defined in \eqref{rgammaC}.
Hence $\lambda=\frac{1}{n^{1+2 \gamma + \Delta_1}}$ and $\psi(x) = -  \frac{1}{3!2\beta_n K_n}G_{\beta_n,K_n}^{(4)}\left(0\right) x^3$ 
we compare the distribution of $W_{\gamma}$ with a distribution with Lebesgue-density proportional to $\exp \bigl( - \frac{x^4}{4 \E(W^4)} \bigr)$.
Again as in the proofs of part (2) and (3) we see that the leading order is the order of $\lambda^{-1} \sqrt{\E[\tilde{R}_{\gamma}^2]}$. We now have to solve
the minimization problem $\min(\gamma,  \Delta_2 -1 + 2 \gamma, 6 \gamma -1)$ for $\gamma  \in (1/6, 1/4)$.

Proofs of part (6) and (7): With $\gamma = 1/6$ and $\Delta_1=1/3, \Delta_2 >2/3$ or $\Delta_1 >1/3, \Delta_2 =2/3$ we consider the 
the linear regression condition 
$$
\E [ W_{1/6}- W_{1/6}' | \F ] =\frac{1}{3!2\beta_n K_n}G_{\beta_n,K_n}^{(4)}\left(0\right) \frac{1}{n^{5/3}} W_{1/6}^3 +\frac{1}{5! 2 \beta_n K_n}G_{\beta_n,K_n}^{(6)}\left(0\right) \frac{1}{n^{5/3}} W_{1/6}^5 + \tilde{R}_{1/6} 
$$
with $\tilde{R}_{1/6} = T_1+ R_{1/6}$ or
$$
\E [ W_{1/6}- W_{1/6}' | \F ] =  \frac{k}{K_c(\beta_n)} \frac{1}{n^{5/3}} W_{1/6}+\frac{1}{5! 2 \beta_n K_n}G_{\beta_n,K_n}^{(6)}\left(0\right) \frac{1}{n^{5/3}} W_{1/6}^5 + \tilde{R}_{1/6} 
$$
with $\tilde{R}_{1/6} = T_2 + R_{1/6}$. In the first case solve $\min(11/6,  1+ \Delta_2)$, in the second
solve $\min(11/6,  \Delta_1 + 4/3)$.

Proof of part (8): Finally we consider the linear regression identity
$$
\E [ W_{\gamma} - W_{\gamma}' | \F ] = \frac{1}{2\beta_n K_n}G_{\beta_n,K_n}^{(2)}\left(0\right) \frac 1n W_{\gamma}
+\frac{1}{3!2\beta_n K_n}G_{\beta_n,K_n}^{(4)}\left(0\right) \frac{1}{n^{1+2\gamma}} W_{\gamma}^3 + \tilde{R}_{\gamma}
$$
with $\tilde{R}_{\gamma} = T_3 + R_{\gamma}$. Solve  $\min(\gamma,  2 \gamma, 6 \gamma -1)= \min(\gamma, 6 \gamma -1)$ for $\gamma \in (1/6, 1/4)$.
\end{proof}
\addtocontents{toc}{%
  \protect\setcounter{tocdepth}{6}%
}

%\renewcommand{\bibname}{References}
%\bibliographystyle{amsplain}
%\bibliography{10.01.2013}
\newcommand{\SortNoop}[1]{}\def\cprime{$'$} \def\cprime{$'$}
  \def\polhk#1{\setbox0=\hbox{#1}{\ooalign{\hidewidth
  \lower1.5ex\hbox{`}\hidewidth\crcr\unhbox0}}}
\providecommand{\bysame}{\leavevmode\hbox to3em{\hrulefill}\thinspace}
\providecommand{\MR}{\relax\ifhmode\unskip\space\fi MR }
% \MRhref is called by the amsart/book/proc definition of \MR.
\providecommand{\MRhref}[2]{%
  \href{http://www.ams.org/mathscinet-getitem?mr=#1}{#2}
}
\providecommand{\href}[2]{#2}

\end{document}